\documentclass[reqno,letterpaper,11pt]{article}

\usepackage{hyperref}

\usepackage{times}
\usepackage[mathcal]{euscript}

\usepackage{amsmath}
\usepackage{amsfonts,amssymb,amsmath,latexsym,wasysym,mathrsfs,bbm,stmaryrd}
\usepackage[all]{xy}
\usepackage[usenames]{color}
\usepackage{epsfig}

\usepackage{amsthm}
\usepackage{thmtools}
\usepackage{thm-restate}
\declaretheorem[numberwithin=section]{theorem}
\declaretheorem[sibling=theorem]{proposition}
\declaretheorem[sibling=theorem]{definition}
\declaretheorem[sibling=theorem]{corollary}
\declaretheorem[sibling=theorem]{lemma}

\declaretheorem[sibling=theorem]{notation}
\declaretheorem[sibling=theorem,name=Theorem / Definition]{thmdef}
\declaretheorem[sibling=theorem,style=remark]{remark}
\declaretheorem[sibling=theorem,style=remark]{example}

\numberwithin{equation}{section} 

\def\R{\mathbb R}
\def\C{\mathbb C}
\def\Z{\mathbb Z}
\def\N{\mathbb N}

\def\P{\mathbb P}
\def\E{\mathbb E}
\def\Var{\mathrm{Var}}
\def\Cov{\mathrm{Cov}}
\def\u{\mathfrak{u}}
\def\gl{\mathfrak{gl}}
\def\deg{\mathrm{deg}}
\def\nor{\mathrm{nor}}

\def\1{\mathbbm 1}

\def\d{\delta}

\def\p{\varphi}

\def\ex{\varepsilon}

\def\t{\tau}

\def\M{\mathbb{M}}
\def\U{\mathbb{U}}
\def\GL{\mathbb{GL}}
\def\A{\mathscr{A}}
\def\EX{\mathscr{E}}
\def\PP{\mathscr{P}}
\def\D{\mathcal{D}}
\def\L{\mathcal{L}}
\def\CC{\mathcal{C}}

\def\u{\mathfrak{u}}

\def\tensor{\otimes}

\def\del{\partial}

\newcommand{\tr}{\mathrm{tr}}
\newcommand{\Tr}{\mathrm{Tr}}
\newcommand{\mx}[1]{\mathbf{#1}}

\renewcommand\emptyset{\varnothing}

\begin{document}

\title{The Large-$N$ Limits of Brownian Motions on $\mathbb{GL}_N$}
\author{Todd Kemp\thanks{Supported by NSF CAREER Award DMS-1254807} \\
Department of Mathematics \\
University of California, San Diego \\
La Jolla, CA 92093-0112 \\
\texttt{tkemp@math.ucsd.edu}
}

\date{\today}% \quad {\em File:} \jobname{.tex}}

\maketitle

\begin{abstract} We introduce a two-parameter family of diffusion processes $(B_{r,s}^N(t))_{t\ge 0}$, $r,s>0$, on the general linear group $\mathbb{GL}_N$ that are Brownian motions with respect to certain natural metrics on the group.  At the same time, we introduce a two-parameter family of free It\^o processes $(b_{r,s}(t))_{t\ge 0}$ in a faithful, tracial $W^\ast$-probability space, and we prove that the full process $(B^N_{r,s}(t))_{t\ge 0}$ converges to $(b_{r,s}(t))_{t\ge 0}$ in noncommutative distribution as $N\to\infty$ for each $r,s>0$.  The processes $(b_{r,s}(t))_{t\ge 0}$ interpolate between the free unitary Brownian motion when $(r,s)=(1,0)$, and the free multiplicative Brownian motion when $r=s=\frac12$; we thus resolve the open problem of convergence of the Brownian motion on $\mathbb{GL}_N$ posed by Biane in \cite{Biane1997c}.   \end{abstract}

\tableofcontents

\section{Introduction\label{Section Introduction}}

Let $\M_N$ denote the space of $N\times N$ complex matrices, and let $\GL_N$ denoted the Lie group of invertible matrices in $\M_N$; its Lie algebra is the full matrix algebra $\gl_N = \M_N$.  The Lie algebra $\gl_N$ possesses no $\mathrm{Ad}(\GL_N)$-invariant inner product.  By contrast, the Lie group $\U_N=\{U\in\M_N\colon UU^\ast=I_N\}$ of unitary matrices in $\M_N$ is compact, and the Hilbert-Schmidt inner product $\langle \xi,\eta\rangle = -\Tr(\xi\eta)$ is $\mathrm{Ad}(\U_N)$-invariant on the Lie algebra $\u_N =\{\xi\in\M_N\colon \xi^\ast=-\xi\}$.  (If we restrict to $\mathfrak{su}_N$, this is the the {\em unique} $\mathrm{Ad}(\mathbb{SU}_N)$-invariant inner product, up to scale.)  In fact, the Hilbert-Schmidt complex inner product $\langle \xi,\eta\rangle = \Tr(\xi\eta^\ast)$ on $\gl_N$ is also invariant under conjugation by elements of $\U_N$.  %It is close to, but not quite, unique up to scale in this regard; see Section ??? below.

The group $\GL_N$ is the complexification of $\U_N$, which is to say that the Lie algebras satisfy $\gl_N = \u_N\oplus i\u_N$.  Both of the complementary real subspaces $\u_N$ (skew-Hermitian matrices) and $i\u_N$ (Hermitian matrices) are invariant under conjugation by elements of $\U_N$.  It follows immediately that the following {\em real} inner products are all $\mathrm{Ad}(\U_N)$-invariant.

\begin{definition} Let $r,s>0$.  Define the real inner product $\langle\cdot,\cdot\rangle_{r,s}$ on $\gl_N$ by
\begin{equation} \label{e.innprodrs} \langle \xi_1+i\eta_1,\xi_2+i\eta_2\rangle^N_{r,s} = -\frac1r N\Tr(\xi_1\xi_2) -\frac1s N\Tr(\eta_1\eta_2), \qquad \xi_1,\xi_2,\eta_1,\eta_2\in\u_N. \end{equation}
That is: $\langle \cdot,\cdot\rangle^N_{r,s}$ makes $\u_N$ and $i\u_N$ orthogonal, and its restrictions to these two orthocomplementary subspaces are positive scalar multiples of the Hilbert-Schmidt inner product.
\end{definition}

\begin{remark} The inner product $\langle\cdot,\cdot\rangle^N_{r,s}$ may alternatively be written in the form
\[ \langle A,B\rangle = \frac12\left(\frac1s+\frac1r\right)N\Re\Tr(AB^\ast) + \frac12\left(\frac1s-\frac1r\right)N\Re\Tr(AB). \]
We scale with $N\Tr$ in order to produce a meaningful limit as $N\to\infty$.
\end{remark}

Any real inner product on $\gl_N$ gives rise to a left-invariant Riemannian metric on $\GL_N$, and hence to a left-invariant Laplace-Beltrami operator, and associated diffusion process: the Brownian motion.

\begin{definition} \label{d.BM0} Let $r,s>0$.  Let $\Delta_{r,s}^N$ denote the Laplace-Beltrami operator on $\GL_N$ associated to the left-invariant Riemannian metric induced by the inner product $\langle\cdot,\cdot\rangle^N_{r,s}$.  The Markov diffusion $B^N_{r,s}(t)$ on $\GL_N$, started at $B^N_{r,s}(0)=I_N$, with generator $\frac12\Delta^N_{r,s}$, is called a {\bf $\U_N$-invariant Brownian motion}.  Fix a probability space ($\Omega,\mathscr{F},\mathbb{P})$ from which the random matrices $B^N_{r,s}(t)$ are sampled, and denote by $\E = \int_\Omega\cdot\,d\mathbb{P}$.
\end{definition}

\begin{remark} \label{r.UNinvB} Since the inner product is $\mathrm{Ad}_{\U_N}$-invariant, its Laplace operator is also unitarily invariant.  That is: for $f\in C^\infty(\GL_N)$ and $U\in \U_N$, let $(\mathrm{Ad}_U^\ast f)(A) = f(\mathrm{Ad}_UA)$; then $\Delta_{r,s}(\mathrm{Ad}_U^\ast f) = \Delta_{r,s} f$ for all $U\in \U_N$.  It follows that the law of the Brownian motion $B_{r,s}^N(t)$ (the heat kernel) is also invariant under the $\mathrm{Ad}^\ast$-action of $\U_N$; hence, it is appropriate to call it $\U_N$-invariant Brownian motion.
\end{remark}

For convenience, we now fix a (large) probability space $(\Omega,\mathscr{F},\P)$ on which all of the random matrices $\{B^N_{r,s}(t)\colon r,s>0,t\ge 0,N\in\N\}$ live.  As usual, for random variables $F$ on $(\Omega,\mathscr{F}$), we denote $\int_\Omega F\,d\P = \E(F)$.  We will characterize the large-$N$ limit of $B^N_{r,s}(t)$ as a noncommutative stochastic process.  To do so, we introduce the following free stochastic processes (for a discussion of free stochastic calculus, see Section \ref{section free stochastic calculus}).

\begin{definition} \label{d.fibm} Fix $r,s\ge 0$.  Let $(\A,t)$ be a $W^\ast$-probability space that contains two freely independent free semicircular Brownian motions $x(t),y(t)$.  Let
\begin{equation} \label{e.w} w_{r,s}(t) = i\sqrt{r}\,x(t) + \sqrt{s}\,y(t). \end{equation}
The {\bf free multiplicative Brownian motion} of parameters $r,s$, denoted $b_{r,s}(t)$, is the unique solution to the following free stochastic differential equation (fSDE):
\begin{equation} \label{e.fSDE0} db_{r,s}(t) = b_{r,s}(t)\,dw_{r,s}(t) - \frac12(r-s)b_{r,s}(t)\,dt, \quad b_{r,s}(0) = 1. \end{equation}
\end{definition}

Let $\tr$ denote the {\em normalized} trace, $\tr = \frac1N\Tr$ on $\M_N$. The main theorem of this paper is as follows; it is proved in Section \ref{section Convergence of the Brownian Motions}.

\begin{theorem} \label{t.main} For $r,s>0$, the Brownian motion $(B^N_{r,s}(t))_{t\ge 0}$ on $\GL_N$ converges, as a noncommutative stochastic process, to $(b_{r,s}(t))_{t\ge 0}$ as $N\to\infty$.  That is to say: if $n\in\N$, $t_1,\ldots,t_n\ge 0$, and $\ex_1,\ldots,\ex_n\in\{1,\ast\}$, then
\[ \lim_{N\to\infty} \E\tr\!\left(B^N_{r,s}(t_1)^{\ex_1} \cdots B^N_{r,s}(t_n)^{\ex_n}\right) = \t\!\left(b_{r,s}(t_1)^{\ex_1} \cdots b_{r,s}(t_n)^{\ex_n}\right). \]
\end{theorem}
\noindent This theorem resolves a conjecture left open by Biane in \cite{Biane1997c}.  Indeed, let $G^N(t) = B^N_{1/2,1/2}(t)$, and let $g(t) = b_{1/2,1/2}(t)$; then (\ref{e.fSDE0}) becomes
\begin{equation} \label{e.fmbm} dg(t) = g(t)\,dw(t), \qquad g(0)=1, \end{equation}
where $w(t) = (y(t)+ix(t))/\sqrt{2}$ is a free circular Brownian motion.  The process $g(t)$ is referred to as {\em free multiplicative Brownian motion} in \cite{Biane1997c,Biane1997b}, where it was conjectures that $(G^N(t))_{t\ge0}$ converges to $(g(t))_{t\ge 0}$ as a noncommutative stochastic process.  Recent progress on this conjecture was made by Guillaume C\'ebron in  \cite[Theorem 4.6]{Guillaume2013}, where he showed that, for each fixed $t\ge 0$, the random matrix $G^N(t)$ converges in noncommutative distribution to $g(t)$. At the same time, the present author in \cite{Kemp2013a} independently proved that, for each fixed $t\ge 0$, the empirical  noncommutative distribution of the random matrix $B^N_{r,s}(t)$ converges almost surely to a linear functional $\p_{r,s}(t)\colon\C\langle X,X^\ast\rangle\to\C$, which is the noncommutative distribution of an operator in a tracial noncommutative probability space;
%; moreover, for any noncommutative polynomial $f\in\C\langle X,X^\ast\rangle$,
%\begin{equation} \label{e.var0} \Var\left[\tr\!\left(f(B^N_{r,s}(t),B^N_{r,s}(t)^\ast)\right)\right] \le \frac{C_{r,s}(t,f)}{N^2} \end{equation}
%for some constant $C_{r,s}(t,f)$ that depends continuously on $r,s,t$.
it was left open whether the trace is faithful. Theorem \ref{t.main} resolves this question as well.  Our present techniques are quite different from those in \cite{Guillaume2013,Kemp2013a}.

\begin{remark} We may also consider the ``special case'' $(r,s)=(1,0)$.  Let $u(t) = b_{1,0}(t)$; then (\ref{e.fSDE0}) becomes
\begin{equation} \label{e.fubm} du(t) = iu(t)\,dx(t) -\frac12u(t)\,dt, \qquad u(0)=1 \end{equation}
which is the fSDE for the (left) {\em free unitary Brownian motion}, introduced in \cite{Biane1997c}.  The main theorem \cite[Theorem 1]{Biane1997c} of that paper was the convergence of the Brownian motion $(U^N(t))_{t\ge 0}$ on $\U_N$ (with respect to to the inner product $-N\Tr(\xi\eta)$) to $(u(t))_{t\ge 0}$.  Some of the ideas we present here are motivated by this example.\end{remark}

In order to prove Theorem \ref{t.main}, we need to describe more concretely the noncommutative distribution of $b_{r,s}(t)$; to that end, we introduce the following indispensable constants.

\begin{thmdef}[\cite{Bercovici1992,Biane1997c}] \label{d.fubm} For each $t\in\R$, there exists a unique probability measure $\nu_t$ on $\C^\ast=\C\setminus\{0\}$ with the following properties.  For $t>0$, $\nu_t$ is supported in the unit circle $\U$; for $t<0$, $\nu_t$ is supported in $\R_+=(0,\infty)$; and $\nu_0=\delta_1$.  In all cases, $\nu_t$ is determined by its moments: $\nu_0(t)\equiv 1$ and, for $n\in\Z\setminus\{0\}$,
\begin{equation} \label{e.nuNu0} \nu_n(t) \equiv \int_{\C^\ast} u^n\,\nu_t(du) = e^{-\frac{|n|}{2}t}\sum_{k=0}^{|n|-1} \frac{(-t)^k}{k!}|n|^{k-1}\binom{|n|}{k+1}. \end{equation}
\end{thmdef}

%The fSDE (\ref{e.fSDE0}) is a very useful computational tool in determining the limit noncommutative distribution of $B^N_{r,s}(t)$.  We calculate several classes of $\ast$-moments, in terms of the following quantities

\begin{theorem} \label{t.moments} Let $r,s,t\ge 0$ and $n\in\N$.  Then
\begin{align} \label{e.m1} &\t\left[b_{r,s}(t)^n\right] = \t\left[b_{r,s}(t)^{\ast n}\right] = \nu_n((r-s)t), \\ \label{e.m2}
&\t\left[(b_{r,s}(t)b_{r,s}(t)^\ast)^n\right] =  \nu_n(-4st), \\ \label{e.m3}
&\t\left[b_{r,s}(t)^2b_{r,s}(t)^{\ast 2}\right] = e^{4st} + 4st(1+st)e^{(3s-r)t}.
\end{align}
\end{theorem}
\noindent Theorem \ref{t.moments} is proved in Section \ref{section Moment Calculations}.

\begin{remark} Equations (\ref{e.m1}) and (\ref{e.m2}) were proved in the author's paper \cite[Theorems 1.3 \& 1.5]{Kemp2013a}.  They are included here to show how they can be derived more directly from the limit process $b_{r,s}(t)$.  Equation (\ref{e.m3}) will be needed in the proof of Theorem \ref{t.brst} below.
\end{remark}

In Section \ref{section Properties of the Brownian Motions}, we demonstrate that the process $(b_{r,s}(t))_{t\ge 0}$ inherits all of the invariant properties from $B^N_{r,s}(t)$ that qualify it as a Brownian motion.

\begin{theorem} \label{t.brst} For $r,s>0$ and $N\in\N^\ast$, the $\GL_N$ Brownian motion $(B^N_{r,s}(t))_{t\ge 0}$ has independent, stationary multiplicative increments.  If $N\ge 2$, then, with probability $1$, $B^N_{r,s}(t)$ is not a normal matrix for any $t>0$.

For $r,s\ge 0$, the free multiplicative Brownian motion $(b_{r,s}(t))_{t\ge 0}$ is invertible for all $t\ge0$, and has freely independent, stationary multiplicative increments.  If $s=0$, then $u(t)$ is unitary, and $u(t)\equiv b_{r,0}(t/r)$ is a free unitary Brownian motion for any $r>0$.  If $s>0$, then $b_{r,s}(t)$ is not a normal operator for any $t>0$.
\end{theorem}

\begin{remark} \label{r.B=U} We defined $B_{r,s}^N(t)$ only for $r,s>0$ (indeed, the inner product $\langle\cdot,\cdot\rangle^N_{r,s}$ blows up as $r\to 0$ or $s\to 0$).  In the case $s=0$, it is possible to make sense of $B_{r,0}^N(t)$ as the solution to the matrix SDE (\ref{e.mSDE-B}) below.  In this case, the process is degenerate on $\GL_N$; in fact, $B_{r,0}^N(t)\in\U_N$, and $U^N(t) = B^N_{r,0}(t/r)$ is Brownian motion on $\U_N$, as in the large-$N$ limit.
\end{remark}

The proof of Theorem \ref{t.main} has two main parts: first, we show that $B^N_{r,s}(t)$ converges to $b^N_{r,s}(t)$ in noncommutative distribution for each fixed $t\ge 0$.  We then use Theorem \ref{t.brst}: since the increments of $(b_{r,s}(t))_{t\ge 0}$ are freely independent, to prove convergence of the process it suffices to prove that the increments of $(B^N_{r,s}(t))_{t\ge 0}$ are {\em asymptotically free}.  The key to proving this property is the following multivariate extension of the technology in \cite[Sections 3 \& 4]{DHK2013}.

\begin{theorem} \label{t.multi-cov} Let $n\in\N$, $t_1,\ldots,t_n\ge 0$, and let $B^{1,N}_{r,s}(t_1),\ldots,B^{n,N}_{r,s}(t_n)$ be independent copies of the Brownian motion $B^N_{r,s}(\cdot)$ at these times.  These operators possess a limit joint distribution, and, for any noncommutative polynomials $f,g\in\C\langle X_1,\ldots,X_n,X_1^\ast,\ldots,X_n^\ast\rangle$, there is a constant $C = C(r,s,t_1,\ldots,t_n,f,g)$ such that
\begin{equation} \label{e.multi-cov} \Cov\!\left[\tr(f(B^{1,N}_{r,s}(t_1),\ldots,B^{1,N}_{r,s}(t_n)^\ast)),\tr(g(B^{1,N}_{r,s}(t_1),\ldots,B^{n,N}_{r,s}(t_n)^\ast))\right] \le \frac{C}{N^2}. \end{equation}
\end{theorem}
\noindent Theorem \ref{t.multi-cov} is proved in Section \ref{section Heat Kernels}.

\section{Background}

In this section, we briefly outline the technology needed to prove the results in this paper: matrix stochastic calculus (particularly for invertible random matrices), the corresponding stochastic calculus in the free probability setting, and the notion of {\em asymptotic freeness} that ties the two together.

\subsection{Stochastic Calculus on $\GL_N$} Let $G$ be a Lie group, with Lie algebra $\mathfrak{g}$.  For $\xi\in\mathfrak{g}$, the associated left-invariant vector field on $G$ is denoted $\del_\xi$:
\begin{equation} \label{e.delxi} \left(\del_\xi f\right)(g)  = \left.\frac{d}{dt}\right|_{t=0} f(g\exp(t\xi)), \qquad f\in C^\infty(G). \end{equation}
Let $\langle\cdot,\cdot\rangle$ be a real inner product on $\mathfrak{g}$, and let $\beta$ be an orthonormal basis for $(\mathfrak{g},\langle\cdot,\cdot\rangle)$.  Then the Laplace-Beltrami operator on $G$ for the Riemannian metric induced by $\langle\cdot,\cdot\rangle$ is
\begin{equation} \label{e.DeltaG} \Delta_G = \sum_{\xi\in\beta} \del_\xi^2, \end{equation}
which does not depend on the particular orthonormal basis used.

If $G\subset\M_N$ is a linear Lie group, then the Brownian motion on $G$ (the diffusion process with generator $\frac12\Delta_G$) may be constructed as the solution to a matrix stochastic differential equation (mSDE).  Fix an orthonormal basis $\beta$ for $\mathfrak{g}$, and let $W(t)$ denote the following Wiener process in $\mathfrak{g}$:
\[ W(t) = \sum_{\xi\in\beta} W_{\xi}(t)\,\xi, \]
where $\{W_\xi\colon \xi\in\beta\}$ are i.i.d.\ standard $\R$-valued Brownian motions.  Then the Brownian motion $B(t)$ is determined by the Stratonovich mSDE
\begin{equation} \label{e.Strat} dB(t) = W(t)\circ dW(t), \qquad W(0) = I_N. \end{equation}
While convenient for proving geometric invariance, the Stratonovich form is less well-adapted to computation.  We can convert (\ref{e.Strat}) to It\^o form.  The result, due to McKean \cite[p.\ 116]{McKean1969} is
\begin{equation} \label{e.B(t)Ito} dB(t) = B(t)\,dW(t) + \frac12B(t)\left(\sum_{\xi\in\beta} \xi^2\right)\,dt, \qquad B(0)=I_N. \end{equation}
See, also, \cite{Gordina2003}.

Let us specialize to the case of interest, with $G=\GL_N$ and $\gl_N$ equipped with an $\mathrm{Ad}_{\U_N}$-invariant inner product $\langle\cdot,\cdot\rangle_{r,s}^N$ of (\ref{e.innprodrs}).  To clarify: let $\langle\cdot,\cdot\rangle_{\u_N}$ denote the following real inner product on $\u_N$:
\begin{equation} \label{e.innproduN} \langle \xi,\eta\rangle_{\u_N} = -N\Tr(\xi\eta). \end{equation}
Then the inner product $\langle\cdot,\cdot\rangle_{r,s}^N$ on $\gl_N=\u_N\oplus i\u_N$ is given by
\begin{equation} \label{e.innprodrs2} \langle \xi_1+i\eta_1,\xi_2+i\eta_2\rangle_{r,s}^N = \frac{1}{r}\langle \xi_1,\xi_2 \rangle_{\u_N} + \frac{1}{s}\langle \eta_1,\eta_2\rangle_{\u_N}. \end{equation}
It is straightforward to check that, if $\beta_N$ is an orthonormal basis for $\u_N$ with respect to $\langle\cdot,\cdot\rangle_{\u_N}$, then
\begin{equation} \label{e.betars} \beta^N_{r,s} = \left\{\sqrt{r}\xi\colon \xi\in \beta_N\right\}\cup\left\{\sqrt{s}i\xi\colon \xi\in\beta_N\right\} \end{equation}
is an orthonormal basis for $\gl_N$ with respect to $\langle\cdot,\cdot\rangle_{r,s}^N$.  Equation (\ref{e.DeltaG}) and a straightforward application of the chain rule in (\ref{e.delxi}) then shows that the Laplace-Beltrami operator is
\begin{equation} \label{e.DeltarsN} \Delta_{r,s}^N = \sum_{\xi\in\beta} (r\del_\xi^2 + s\del_{i\xi}^2). \end{equation}

\begin{remark} In \cite{DHK2013,Kemp2013a}, we used the elliptic operator
\[ A_{s,t}^N = \left(s-\frac{t}{2}\right)\sum_{\xi\in\beta_N} \del_\xi^2 + \frac{t}{2}\sum_{\xi\in\beta_N} \del_{i\xi}^2 = \Delta^N_{s-t/2,t/2}. \]
The linear change of parameters was convenient for our discussion of the two-parameter Segal--Bargmann transform, and so all of the theorems in \cite{Kemp2013a} are stated using this language as well.  \end{remark}

In \cite[Proposition 3.1]{DHK2013}, the following ``magic formula'' was proved.  If $\beta_N$ is an orthonormal basis of $\u_N$, then
\begin{equation} \label{e.magic} \sum_{\xi\in\beta_N} \xi^2 = -I_N. \end{equation}
Combining this with (\ref{e.betars}) gives
\[ \sum_{\xi\in\beta^N_{r,s}} \xi^2 = -(r-s)I_N, \]
and so, by (\ref{e.B(t)Ito}), the $\U_N$-invariant Brownian motion $B^N_{r,s}(t)$ is determined by the mSDE
\begin{equation} \label{e.mSDE-B} dB^N_{r,s}(t) = B^N_{r,s}(t)\,dW^N_{r,s}(t) -\frac12(r-s)B^N_{r,s}(t)\,dt, \end{equation}
where $W^N_{r,s}(t) = \sum_{\xi\in\beta^N_{r,s}} W_\xi(t)\,\xi$.  It will be convenient to express this It\^o process in a slightly different form. Let us choose the following orthonormal basis $\beta_N$ for $\u_N$:
\begin{equation} \label{e.beta_N} \beta_N = \left\{\frac{1}{\sqrt{N}}E_{jj}, \frac{1}{\sqrt{2N}}(E_{jk}-E_{kj}),\frac{1}{\sqrt{2N}}i(E_{jk}+E_{jk})\colon 1\le j<k\le N\right\}, \end{equation}
where $E_{jk}$ is the matrix unit with a $1$ in the $(j,k)$-entry and $0$ elsewhere.  Then it is strightforward to check that
\[ W^N_{r,s}(t) = \sqrt{r}\sum_{\xi\in \beta_N} B_\xi(t)\,\xi + i\sqrt{s} \sum_{\xi\in \beta_N} B_{i\xi}(t)\,\xi = \sqrt{r}\, i X^N(t) + \sqrt{s}\, Y^N(t), \]
where $X^N(t)$ and $Y^N(t)$ are independent {\em $\mathrm{GUE}_N$ Brownian motions}.  That is: all entries of $X^N(t)$ are independent from all entries of $Y^N(t)$; the matrices $X^N(t),Y^N(t)$ are Hermitian; and all entries $[X^N(t)]_{jk}$ and $[Y^N(t)]_{jk}$ with $1\le j\le k\le N$ are i.i.d.\ $\R$-valued Brownian motions of variance $t/N$.  This is a convenient representation, due to the following easily-verified stochastic calculus rules that apply to matrix stochastic integrals with respect to (linear combinations of) $X^N(t)$ and $Y^N(t)$.

\begin{lemma} \label{l.mIto} Let $\Theta(t),\Theta_1(t),\Theta_2(t)$ be $\M_N$-valued stochastic processes that are adapted to the filtration $\mathscr{F}_t$ of $X^N(t)$ and $Y^N(t)$ (in the probability space $(\Omega,\mathscr{F},\mathbb{P})$).  Then the following hold:
\begin{align} \label{e.mIto1} &\E(\Theta_1(t)\,dX^N(t)\,\Theta_2(t)) = \E(\Theta_1(t)\,dY^N(t)\,\Theta_2(t)) = 0 \\
&\label{e.mIto2} dX^N(t)\,\Theta(t)\,dX^N(t) = dY^N(t)\,\Theta(t)\,dY^N(t) = \tr(\Theta(t))I_N\,dt \\
&\label{e.mIto3} dX^N(t)\,\Theta(t)\,dY^N(t) = dY^N(t)\,\Theta(t)\,dX^N(t) = 0 \\
&\label{e.mIto4} \Theta_1(t)\,dX^N(t)\,\Theta_2(t)\,dt = \Theta_1(t)\,dY^N(t)\,\Theta_2(t)\,dt = 0.
\end{align}
Moreover, let $\Theta_1(t)$ and $\Theta_2(t)$ be $\M_N$-valued It\^o processes: solutions to mSDEs of the form
\[ d\Theta(t) = f_1(\Theta(t))\,dX^N(t)\,f_2(\Theta(t)) + g_1(\Theta(t))\,dY^N(t)\,g_2(\Theta(t)) + h(\Theta(t))\,dt, \]
for smooth functions $f_1,f_2,g_1,g_2,h\colon\M_N\to\M_N$.  Then the following It\^o product rule holds:
\begin{equation} \label{e.mItoprod} d(\Theta_1(t)\Theta_2(t)) = d\Theta_1(t)\cdot \Theta_2(t) + \Theta_1(t)\cdot d\Theta_2(t) + d\Theta_1(t)\cdot d\Theta_2(t). \end{equation}
\end{lemma}
\begin{remark} As usual, we abuse notation and write stochastic integral equations in differential form.  For example, the last equality in \ref{e.mIto1} is shorthand for
\[ \E\left(\int_0^t \Theta_1(s)\, dY^N(s)\,\Theta_2(s)\right) = 0, \]
where the matrix stochastic integral is defined exactly as the scalar stochastic integral, using matrix multiplication in the place of scalar multiplication. Lemma \ref{l.mIto} is straightforward to verify from the standard It\^o calculus for vector-valued processes.
\end{remark}

\subsection{Free Stochastic Calculus \label{section free stochastic calculus}}

For an introduction to noncommutative probability theory, and free probability in particular, we refer the reader to \cite{NicaSpeicherBook}.  We assume familiarity with noncommutative probability spaces and $W^\ast$-probability spaces.    The reader is directed to \cite[Sections 1.1--1.3]{Kemp2012a} for a quick introduction to free additive (semicircular) Brownian motion.  Also, we give a brief discussion of free independence at the beginning of Section \ref{section Asymptotic Freeness} below.

Let $(\A,\t)$ be a faithful, tracial $W^\ast$-probability space.  To fix notation, for $a\in\A$ denote its noncommutative distribution as $\p_a$.  I.e.\ letting $\C\langle X,X^\ast\rangle$ denote the noncommutative polynomials in two variables, $\p_a\colon\C\langle X,X^\ast\rangle\to\C$ is the linear functional
\[ \p_a(f) = \t(f(a,a^\ast)), \qquad f\in\C\langle X,X^\ast\rangle. \]
A {\em free semicircular Brownian motion} $x(t)$ is a self-adjoint stochastic process $\left(x(t)\right)_{t\ge 0}$ in $\A$ such that $x(0)=0$, $\Var(x(1)) = 1$, and the additive increments of $x$ are stationary and freely independent: for $0\le t_1<t_2<\infty$,\break $\p_{x(t_2)-x(t_1)} = \p_{x(t_2-t_1)}$, and $x(t_2)-x(t_1)$ is freely independent from the $W^\ast$-subalgebra $\A\supset \A_{t_1} \equiv W^\ast\{x(t)\colon 0\le t\le t_1\}$.  Since $x(t)$ is a bounded self-adjoint operator, its distribution is given by a compactly-supported probability measure on $\R$; the freeness of increments and stationarity then implies that $\p_{x(t_2)-x(t_1)}$ is the {\em semicircle law}: setting $t=t_2-t_1$,
\[ \t[(x(t_2)-x(t_1))^n] = \int_{-2\sqrt{t}}^{2\sqrt{t}} s^n\,\frac{1}{2\pi t}\sqrt{4t-s^2}\,ds, \qquad n\in\N. \]
In \cite{Voiculescu1991}, it was proven that, if $X^N(t)$ is a $\mathrm{GUE}_N$ Brownian motion, then the {\em process} $(X^N(t))_{t\ge 0}$ converges to a free semicircular Brownian motion: for any $n$ and any $t_1,t_2,\ldots,t_n\ge 0$, and any noncommutative polynomial $f\in\C\langle X_1,\ldots,X_n\rangle$,
\[ \lim_{N\to\infty}\E\tr(f(X^N(t_1),\ldots,X^N(t_n))= \t(f(x(t_1),\ldots,x(t_n)). \]
Appealing to Lemma \ref{l.mIto}, this paves the way to {\em free stochastic differential equations}.

Let $x(t)$ and $y(t)$ be two freely independent free semicircular Brownian motions in a $W^\ast$-probability space $(\A,\t)$, and let $\A_t = W^\ast\{x(s),y(s)\colon 0\le s\le t\}$.  Let $\theta(t),\theta_1(t),\theta_2(t)$ be processes that are adapted to the filtration $\A_t$.  The {\em free It\^o integral} 
\[ \int_0^t \theta_1(s)\,dx(s)\,\theta_2(s) \]
is defined in precisely the same manner as It\^o integrals of real-valued processes with respect to real Brownian motion: as $L^2(\A_t,\t)$-limits of sums $\sum_j \theta_1(t_j) (x(t_j)-x(t_{j-1}))\theta_2(t_j)$ over partitions $\{0=t_0 \le \cdots\le t_n=t\}$ as the partition width $\sup_j|t_j-t_{j-1}|$ tends to $0$.  Standard Picard iteration techniques show that, if $f_1,f_2,g_1,g_2,h$ are polynomials then the integral equation
\begin{equation} \label{e.fSIE} b(t) = 1 + \int_0^t f_1(b(s))\,dx(s)\,f_2(b(s)) + \int_0^t g_1(b(s))\,dy(s)\,g_2(b(s)) + \int_0^t h(b(s))\,ds, \\
\end{equation}
has a unique adapted solution $b(t)\in\A_t$ satisfying $b(0)=1$.  As usual, we use differential notation to express (\ref{e.fSIE}) in the form
\begin{equation} \label{e.fSDE} db(t) =  f_1(b(t))\,dx(t)\,f_2(b(t)) + g_1(b(t))\,dy(t)\,g_2(b(t)) + h(b(t))\,dt, \qquad b(0)=1. \end{equation}
We refer to (\ref{e.fSDE}) as a {\em free stochastic differential equation} (fSDE).  Solutions of such equations are called {\em free It\^o processes}.  The matrix stochasic calculus of Lemma \ref{l.mIto} has a precise analogue for free It\^o processes.

\begin{lemma} \label{l.fIto} Let $(\A,\t)$ be a $W^\ast$-probability space containing two freely independent free semicircular Brownian motions $x(t)$ and $y(t)$, adapted to the filtration $\{\A_t\}_{t\ge 0}$.  Let $\theta(t),\theta_1(t),\theta_2(t)$ be processes adapted to $\A_t$.  Then the following hold:
\begin{align} \label{e.fIto1} &\t(\theta_1(t)\,dx(t)\,\theta_2(t)) = \t(\theta_1(t)\,dy(t)\,\theta_2(t))= 0 \\
\label{e.fIto2} &dx(t)\,\theta(t)\,dx(t) = dy(t)\,\theta(t)\,dy(t) = \t(\theta(t))\,dt \\
\label{e.fIto3} &dx(t)\,\theta(t)\,dy(t) = dy(t)\,\theta(t)\,dx(t) = 0 \\
\label{e.fIto4} &\theta_1(t)\,dx(t)\,\theta_2(t)\,dt = \theta_1(t)\,dy(t)\,\theta_2(t)\,dt = 0.
\end{align}
Moreover, if $\theta_1(t)$ and $\theta_2(t)$ are free It\^o processes, then the following It\^o product rule holds:
\begin{equation} \label{e.fItoprod} d(\theta_1(t)\theta_2(t)) = d\theta_1(t)\cdot\theta_2(t) + \theta_1(t)\cdot d\theta_2(t) + d\theta_1(t)\cdot d\theta_2(t). \end{equation}
\end{lemma}
\noindent For a proof of Lemma \ref{l.fIto}, see \cite{Biane1998a}.

\subsection{Asymptotic Freeness \label{section Asymptotic Freeness}}

%For an introduction to noncommutative probability theory, and free probability in particular, we refer the reader to \cite{NicaSpeicherBook}.  We assume familiarity with noncommutative probability spaces and $W^\ast$-probability spaces.  We include the definition of freeness here, for completeness.

\begin{definition} \label{d.free} Let $(\A,\t)$ be a noncommutative probability space.  Unital $\ast$-subalgebras $\A_1,\ldots,\A_m\subset\A$ are called {\bf free} with respect to $\t$ if, given any $n\in\N$ and $k_1,\ldots,k_n\in\{1,\ldots,m\}$ such that $k_{j-1}\ne k_j$ for $1<j\le n$,  and any elements $a_j\in \A_{k_j}$ with $\t(a_j)=0$ for $1\le j\le n$, it follows that $\t(a_1\cdots a_n)=0$.  Random variables $a_1,\ldots,a_m$ are said to be {\bf freely independent} of the unital $\ast$-algebras $\A_j = \langle a_j,a_j^\ast\rangle\subset \A$ they generate are free.\end{definition}
Free independence is a $\ast$-moment factorization property.  By centering $a_i-\t(a_i)1_\A\in\A_i$, the freeness rule allows (inductively) any moment $\t(a_{k_1}^{\ex_1} \cdots a_{k_n}^{\ex_n})$ to be decomposed as a polynomial in moments $\t(a_i^\ex)$ in the variables separately.  For example, if $a,b$ are freely independent then $\t(a^{\ex} b^{\d})= \t(a^{\ex})\t(b^{\d})$, while
\[ \t(a^{\ex_1}b^{\d_1}a^{\ex_2} b^{\d_2}) = \t(a^{\ex_1})\t(a^{\ex_2})\t(b^{\d_1}b^{\d_2})+ \t(a^{\ex_1}a^{\ex_2})\t(b^{\d_1})\t(b^{\d_2}) - \t(a^{\ex_1})\t(a^{\ex_2})\t(b^{\d_1})\t(b^{\d_2}), \]
for any $\ex,\ex_1,\ex_2,\d,\d_1,\d_2\in\{1,\ast\}$.  In general, if $a_1,\ldots,a_n$ are freely independent, then their noncommutative joint distribution $\p_{a_1,\ldots,a_n}$ (a linear functional on $\C\langle X_1,\ldots,X_n,X_1^\ast,\ldots,X_n^\ast\rangle$) is determined by the individual distributions $\p_{a_1},\ldots,\p_{a_n}$ (linear functionals on $\C\langle X,X^\ast\rangle$).  %We will make the precise mechanism a little clearer in Section \ref{subsection Asymp Freeness}.

Let $L^{\infty-}(\Omega,\mathscr{F},\P) = \bigcap_{p>1} L^p(\Omega,\mathscr{F},\P)$, and let $\M_N\tensor L^{\infty-}$ denote the algebra of $N\times N$ matrices with entries in $L^{\infty-}(\Omega,\mathscr{F},\P)$.  There are no non-trivial instances of free independence in the noncommutative probability space $(\M_N\tensor L^{\infty-},\E\tr)$; i.e.\ if $A,B\in\M_N\tensor L^{\infty-}$ are freely independent, then at least one of them is a.s.\ a constant multiple of the identity matrix $I_N$.  However, {\em asymptotic freeness} abounds.

\begin{definition} \label{d.asympfree} let $n\in\N$.  For each $N\in\N$, let $A_1^N,\ldots,A^N_n$ be random matrices in $\M_N\tensor L^{\infty-}$.  Say that $(A^N_1,\ldots,A^N_n)$ are {\bf asymptotically free} if there is a noncommutative probability space $(\A,\t)$ containing freely independent random variables $a_1,\ldots,a_n$ such that $(A^N_1,\ldots,A^N_n)$ converges in noncommutative distribution to $(a_1,\ldots,a_n)$.
\end{definition}
The general mantra for producing asymptotically free random matrices is as follows.
\begin{quote} If $A_1^N,\ldots,A_n^N$ are random matrices whose distribution is invariant under unitary conjugation, and possess a joint limit distribution, then they are asymptotically free.
\end{quote}
The first result in this direction was proved in \cite{Voiculescu1991}, where the matrices $A_j^N$ were taken to have the form $A_j^N = U^N_j D^N_j (U^N_j)^{-1}$ where $U_1^N,\ldots,U_n^N$ are independent Haar-distributed unitaries, and $D^N_j$ are deterministic diagonal matrices with uniform bounds on their trace moments.  This was later improved to include all deterministic matrices (with uniform bounds on their operator norms) in \cite{Voiculescu1998}; see, also, \cite{Collins2003,Xu1997} for related results.  We will use the following form of the mantra, which is a weak form of \cite[Theorem 1]{Mingo2007}.

\begin{theorem} \label{t.asympfree} Let $A^N_1,\ldots,A^N_n$ be independent random matrices in $\M_N\tensor L^{\infty-}$, with the following properties.
\begin{itemize}
\item[(1)] The joint law of $A^N_1,\ldots,A^N_n$ is invariant under conjugation by unitary matrices in $\U_N$.
\item[(2)] There is a joint limit distribution: for any noncommutative polynomial $f\in\C\langle X_1,\ldots,X_n,X_1^\ast,\ldots,X_n^\ast\rangle$, $\lim_{N\to\infty}\E\tr(f(A^N_1,\ldots,A^N_n,(A^N_1)^\ast,\ldots,(A^N_n)^\ast))$ exists.
\item[(3)]  The fluctuations are $O(1/N^2)$: for any noncommutative polynomials $f,g$ as in (2), there is a constant $C=C(f,g)$ so that
\[ \Cov\!\left[\tr\!\left(f(A^N_1,\ldots,A^N_n,(A^N_1)^\ast,\ldots,(A^N_n)^\ast)\right),\tr\!\left(g(A^N_1,\ldots,A^N_n,(A^N_1)^\ast,\ldots,(A^N_n)^\ast)\right)\right] \le \frac{C}{N^2}. \]
\end{itemize}
Then $A^N_1,\ldots,A^N_n$ are asymptotically free.
\end{theorem}

\begin{remark} \cite[Theorem 1]{Mingo2007} has a much stronger assumption than (3): it also assumes that the classical cumulants $k_r$ in normalized traces of noncommutative polynomials are $o(1/N^r)$ for all $r>2$, thus producing a so-called {\em second-order limit distribution}.  However, this stronger assumption is used only to produce a stronger conclusion: that the matrices are asymptotically free {\em of second-order}.  Following the proof, it is relatively easy to see that Theorem \ref{t.asympfree} is proved along the way, at least in the case $n=2$.  To go from $2$ to general finite $n$ can be achieved by induction together with the associativity of freeness; cf.\ \cite[Proposition 2.5.5(iii)]{VDNBook}.  See, also, \cite{Mingo2013} where this is proved more explicitly in the harder case of real random matrices (where $\U_N$-invariance is replaced with $\mathbb{O}_N$-invariance).
\end{remark}

\section{Heat Kernels on $\GL_N^n$ \label{section Heat Kernels}}

Here we generalize the technology we developed in \cite[Sections 3.4 \& 4.1]{DHK2013} to independent products of heat kernel measures on $\GL_N$.

\subsection{Laplacians on $\GL_N^n$}

Let $n,N\in\N$.  Then $\GL_N^n = \GL_N\times\cdots\times\GL_N$ is a Lie group of real dimension $2nN^2$.  Its Lie algebra is $\gl_N^n = \gl_N\oplus\cdots\oplus\gl_N$.  For $\xi\in\gl_N$, and $1\le j\le n$, let $\xi_j$ denote the vector $(0,\ldots,0,\xi,0,\ldots,0)\in \gl_N^n$ (with $\xi$ in the $j$th component).  The Lie product on $\gl_N^n$ is then determined by $[\xi_j,\eta_k]=\delta_{jk}(\xi_j\eta_k-\eta_j\xi_k)$ for $1\le j,k\le n$.  In particular, if $j\ne k$ and $\xi,\eta\in\gl_N$, then the left-invariant derivations $\del_{\xi_j}$ and $\del_{\eta_k}$ on $C^\infty(\GL_N^n)$ commute.  To be clear, note that, for $f\in C^\infty(\GL_N^n)$,
\begin{equation} \label{e.delxij} (\del_{\xi_j}f)(A_1,\ldots,A_n) = \left.\frac{d}{dt}\right|_{t=0} f(A_1,\ldots,A_{j-1},A_je^{t\xi},A_{j+1},\ldots,A_n). \end{equation}

Let $\beta^N_{r,s}$ denote an orthonormal basis for $\gl_N$ (with respect to $\langle\cdot,\cdot\rangle^N_{r,s}$, as in (\ref{e.betars})).  For $1\le j\le n$, define
\begin{equation} \label{e.Deltaj} \Delta_{r,s}^{j,N} = \sum_{\xi\in\beta^N_{r,s}} \del_{\xi_j}^2. \end{equation}
Note that $\Delta^{j,N}_{r,s}$ and $\Delta^{k,N}_{r,s}$ commute for all $j,k$.  Now, fix $t_1,\ldots,t_n>0$.  Then the operator
\[ t_1\Delta^{1,N}_{r,s} + \cdots + t_n\Delta^{n,N}_{r,s} \]
is elliptic, and essentially self-adjoint on $C^\infty_c(\GL_N^n)$.  We may therefore use the spectral theorem to define the bounded operator
\[ e^{\frac12(t_1\Delta^{1,N}_{r,s} + \cdots + t_n\Delta^{n,N}_{r,s})} = e^{\frac{1}{2}t_1\Delta^{1,N}_{r,s}}\cdots e^{\frac{1}{2}t_n\Delta^{n,N}_{r,s}}. \]
Define the {\bf heat kernel measure} $\mu^{n,N}_{r,s;t_1,\ldots,t_n}$ on $\GL_N^n$ by
\begin{equation} \label{e.hkn} \int_{\GL_N^n} f\,d\mu^{n,N}_{r,s;t_1,\ldots,t_n} = \left(e^{\frac12(t_1\Delta^{1,N}_{r,s} + \cdots + t_n\Delta^{n,N}_{r,s})}f\right)(I_N^n), \qquad f\in C_c(\GL_N^n),
\end{equation}
where $I_N^n = (I_N,\ldots,I_N)\in\GL_N^n$.  In particular, let $K_1,\ldots,K_n\subset\GL_N$ be compact sets; by approximating $\1_{K_1\times\cdots\times K_n}$ with a continuous function, we see that
\[ \mu^{n,N}_{r,s;t_1,\ldots,t_n}(K_1\times\cdots\times K_n) = \left(e^{\frac{1}{2}t_1\Delta^N_{r,s}}\1_{K_1}\right)(I_N) \cdots  \left(e^{\frac{1}{2}t_n\Delta^N_{r,s}}\1_{K_n}\right)(I_N) = \mu^{1,N}_{r,s;t_1}(K_1)\cdots \mu^{1,N}_{r,s;t_n}(K_n). \]
Since $\mu_{r,s;t}^{1,N}$ is the heat kernel measure on $\GL_N$ corresponding to $\Delta^N_{r,s}$, it is the distribution of the Brownian motion $B_{r,s}^N(t)$, and so we have shown the following.

\begin{lemma} \label{l.prodmu} Let $(B^{1,N}_{r,s}(t))_{t\ge 0},\ldots,(B^{n,N}_{r,s}(t))_{t\ge 0}$ be $n$ independent $(r,s)$-Brownian motions on $\GL_N$.  Then the joint law of the random vector $(B^{1,N}_{r,s}(t_1),\ldots,B^{n,N}_{r,s}(t_n))$ is $\mu^{n,N}_{r,s;t_1,\ldots,t_n}$. \end{lemma}

\subsection{Multivariate Trace Polynomials \label{section Trace Polynomials}}

Let $J$ be an index set (for our purposes in this section, we will usually take $J=\{1,\ldots,n\}$ for some $n\in\N$).  Let $\EX_J$ denote the set of all nonempty words in $J\times\{1,\ast\}$, $\EX_J= \bigcup_{n\in\N} (J\times\{1,\ast\})^n$.  Let $\mx{v}_J = \{v_\ex\colon \ex\in\EX_J\}$ be commuting variables, and let
\[ \PP(J) = \C[\mx{v}_J] \]
be the algebra of (commutative) polynomials in the variables $\mx{v}_J$.  That is: as a $\C$-vector space, $\PP(J)$ has as its standard basis $1$ together with the monomials
\begin{equation} \label{e.monomial} v_{\ex^{(1)}}\cdots v_{\ex^{(k)}}, \quad k\in\N, \quad \ex^{(1)},\ldots,\ex^{(k)}\in \EX_J, \end{equation}
and the (commutative) product on $\PP(J)$ is the standard polynomial product.

We may identify monomials in $\C\langle X_j,X_j^\ast\colon j\in J\rangle$ with the variables $v_\ex$, via
\[ \Upsilon(X_{j_1}^{\ex_1}\cdots X_{j_k}^{\ex_k}) = v_{\left((j_1,\ex_1),\ldots,(j_k,\ex_k)\right)}.\]
Extending linearly, $\Upsilon\colon \C\langle X_j,X_j^\ast\colon j\in J\rangle\hookrightarrow \PP(J)$ is a linear inclusion, identifying $\C\langle X_j,X_j^\ast\colon j\in J\rangle$ with the {\em linear} polynomials in $\PP(J)$.  The algebra $\PP(J)$ is the ``universal enveloping algebra'' of $\C\langle X_j,X_j^\ast\colon j\in J\rangle$, in the following sense: any linear functional $\p$ on $\C\langle X_j,X_j^\ast\colon j\in J\rangle$ extends (via $\Upsilon$) uniquely to an algebra homomorphism $\widetilde{\p}\colon \P(J)\to\C$.  Conversely, any algebra homomorphism $\PP(J)\to\C$ is determined by its restriction to $\Upsilon(\C\langle X_j,X_j^\ast\colon j\in J\rangle)$, which intertwines a unique linear functional on $\C\langle X_j,X_j^\ast\colon j\in J\rangle$.  Hence, the noncommutative distribution $\p_{\{a_j\colon j\in J\}}$ of $J$ random variables can be equivalently represented as an algebra homomorphism $\PP(J)\to\C$.

\begin{definition} \label{d.trdeg} For a monomial (\ref{e.monomial}), the {\bf trace degree} is defined to be
\[ \deg(v_{\ex^{(1)}}\cdots v_{\ex^{(k)}}) = |\ex^{(1)}|+\cdots+|\ex^{(k)}|, \]
where $|\ex|=n$ if $\ex\in (J\times\{1,\ast\})^n$.  More generally, if $P\in\PP(J)$, then $\deg(P)$ is the maximal trace degree of the monomial terms in $P$.  Define $\deg(0)=0$.  Note that $\deg(PQ) = \deg(P)+\deg(Q)$, and $\deg(P+Q) \le  \max\{\deg(P),\deg(Q)\}$ for $P,Q\in\PP(J)$.  For $d\in\N$, denote by $\PP_d(J)$ the subspace
\[ \PP_d(J) = \{P\in\PP(J)\colon \deg(P)\le d\}. \]
Note that $\PP_d(J)$ is finite dimensional (if $J$ is finite), and $\PP(J) = \bigcup_{d\ge 1}\PP_d(J)$.
\end{definition}

We now introduce a kind of functional calculus for $\PP(J)$.

\begin{definition} \label{d.fncalc} Let $(\A,\t)$ be a noncommutative probability space.  Let $J$ be an index set, and let $\{a_j\colon j\in J\}$ be specified elements in $\A$.  For $n\in\N$, and $(J\times\{1,\ast\})^n\ni\ex=\left((j_1,\ex_1),\ldots,(j_n,\ex_n)\right)$, define
\[ a^\ex \equiv a_{j_1}^{\ex_1}\cdots a_{j_n}^{\ex_n}. \]
We define for each $P\in\PP(J)$ a complex number $P_\t(a_j\colon j\in J)$ as follows: for $\ex\in\EX_J$, $[v_\ex]_\t(a_j\colon j\in J) = \t(a^\ex)$; and, in general, the map $P\mapsto P_\t(a_j\colon j\in J)$ is an algebra homomorphism from $\PP(J)$ to $\C$.
\end{definition}
\noindent In other words: $P_\t$ is the unique algebra homomorphism extending (via $\Upsilon$) the linear functional $\p_{\{a_j\colon j\in J\}}$ on $\C\langle X_j,X_j^\ast\colon j\in J\rangle$ (i.e.\ the noncommutative distribution of $\{a_j\colon j\in J\}$).

\begin{example} \label{ex.1} Let $J=\{1,2\}$, and consider $\PP(J)\ni P = v_{(1,\ast),(2,1),(1,1)}-2v_{(2,1)}^2$, which has trace degree $3$; then
\[ P_\t(a_1,a_2) = \t(a_1^\ast a_2 a_1)-2\left(\t(a_2)\right)^2. \]
\end{example}
\noindent We generally refer to the functions $\{P_\t\colon P\in\PP(J)\}$ as (multivariate) {\bf trace polynomials}.

\begin{notation} For $N\in\N$, in the noncommutative probability space $(\M_N,\tr)$, we denote the evaluation map $P\mapsto P_{\tr}$ of Definition \ref{d.fncalc} as $P\mapsto P_N$.  Thus, if $A_1,\ldots,A_n\in\M_N\tensor L^{\infty-}$, and $P$ is as in Example \ref{ex.1}, then $P_N(A_1,\ldots,A_n) = \tr(A_1^\ast A_2A_1) -2\left(\tr(A_2)\right)^2$, which is a random variable, to be clear. \end{notation}

\subsection{Intertwining Formula} 

The following ``magic formulas'' appeared as \cite[Proposition 1]{DHK2013}; note that (\ref{e.magic}) is a special case of (\ref{e.ma1}).

\begin{proposition} \label{p.magic} Let $\beta_N$ be an orthonormal basis for $\u_N$ with respect to the inner product (\ref{e.innproduN}).  Then for any $A\in\M_N$
\begin{align} \label{e.ma1} &\sum_{\xi\in\beta_N} \xi A\xi = -\tr(A)I_N, \\ \label{e.ma2}
&\sum_{\xi\in\beta_N} \tr(A\xi)\xi = -\frac{1}{N^2}A. \end{align}
\end{proposition}

For the remainder of this section, we usually suppress the indices $r,s$ for notational convenience; so, for example, $\Delta^{j,N}\equiv \Delta^{j,N}_{r,s}$ for $1\le j \le n$.  Let $J=\{1,\ldots,n\}$ throughout.

\begin{theorem} \label{t.intertwine0} Let $j\in J$.  There are collections $\left\{Q^j_\ex\colon \ex\in\EX_J\right\}$  and $\left\{R^j_{\ex,\d}\colon \ex,\d\in\EX_J\right\}$ in $\PP(J)$ with the following properties.
\begin{itemize}
\item[(1)] For each $\ex\in\EX_J$, $Q^j_\ex$ is a finite sum of monomials of homogeneous trace degree $|\ex|$ such that
\[ \Delta^{j,N}\left([v_\ex]_N\right) = [Q^j_\ex]_N. \]
\item[(2)] For each $\ex,\d\in\EX_J$, $R^j_{\ex,\d}$ is a finite sum of monomials of homogeneous trace degree $|\ex|+|\d|$ such that
\[ r\sum_{\xi\in\beta_N}(\del_{\xi_j} [v_\ex]_N)(\del_{\xi_j} [v_\d]_N) + s\sum_{\xi\in\beta_N} (\del_{i\xi_j}[v_\ex]_N)(\del_{i\xi_j}[v_\d]_N) = \frac{1}{N^2}[R^j_{\ex,\d}]_N, \]
for any orthonormal basis $\beta_N$ of $\u_N$.
\end{itemize}
\end{theorem}
\noindent Please note that $Q^j_\ex$ and $R^j_{\varepsilon,\delta}$ {\em do not depend on $N$}.  The $1/N^2$ in (2) comes from the magic formula (\ref{e.ma2}), as we will see in the proof.

\begin{proof} Fix $\EX_J \ni \ex=((j_1,\ex_1),\ldots,(j_m,\ex_m))$; then $[v_\ex]_N(A_1,\ldots,A_n) = \tr(A_{j_1}^{\ex_1}\cdots A_{j_m}^{\ex_m})$.  Applying the product rule, for any $\xi\in\beta_N$ we have
\begin{align} \label{e.diff1} \del_{\xi_j}^2\left([v_\ex]_N\right) &= \sum_{k=1}^m \delta_{j,j_k} \tr(A_{j_1}^{\ex_1}\cdots (A_{j_k}\xi^2)^{\ex_k} \cdots A_{j_m}^{\ex_m}) \\ \label{e.diff2}
& \qquad + 2\sum_{1\le k<\ell\le m} \delta_{j,j_k}\delta_{j,j_\ell} \tr(A_{j_1}^{\ex_1}\cdots (A_{j_k}\xi)^{\ex_k}\cdots (A_{j_\ell}\xi)^{\ex_\ell} \cdots A_{j_m}^{\ex_m}).
\end{align}
Similarly, $\del_{i\xi_j}^2$ is given by the same formula but possibly with some minus signs in some of the terms (depending on $\ex_k,\ex_\ell$).  For convenience, let $\beta^+_N = \beta_N$ and $\beta^-_N = i\beta_N$.  Magic formula (\ref{e.magic}) gives $\sum_{\xi\in\beta_N^\pm} \xi^2 = \mp I_N$, and so summing over $\beta_N^\pm$ we have, for each $k$,
\[ \sum_{\xi\in\beta_N^\pm} \tr(A_{j_1}^{\ex_1}\cdots (A_{j_k}\xi^2)^{\ex_k} \cdots A_{j_m}^{\ex_m}) = \pm [v_\ex]_N, \]
where the $\pm$ on the left and right do not necessarily match (we will not keep careful track of signs through this proof).  Thus, (\ref{e.diff1}) summed over $\beta_N^\pm$ gives some integer multiple $n_j^\pm(\ex)$ of $[v_\ex]_N$.  Summing the terms in (\ref{e.diff2}) over $\xi\in\beta_N^\pm$, using (\ref{e.ma1}), yields
\[ \sum_{\xi\in\beta_N^\pm} \tr(A_{j_1}^{\ex_1}\cdots (A_{j_k}\xi)^{\ex_k}\cdots (A_{j_\ell}\xi)^{\ex_\ell} \cdots A_{j_m}^{\ex_m}) = \pm [v_{\ex_{k,\ell}}]_N[v_{\ex_{k,\ell}'}]_N, \]
where $\ex_{k,\ell}$ is a substring of $\ex$ (running between index $k$ or $k+1$ and index $\ell-1$ or $\ell$, depending on $\ex_k,\ex_\ell$) and $\ex_{k,\ell}'$ is the concatenation of the two remaining substrings of $\ex$ when $\ex_{k,\ell}$ is removed.  Hence, define
\[ Q^{j,\pm}_{\ex} = n_\pm(\ex)v_\ex + 2\sum_{1\le k<\ell\le m} \pm \delta_{j,j_k}\delta_{j,j_\ell} v_{\ex_{k,\ell}}v_{\ex_{k,\ell}'}.  \]
Note that $|\ex| = |\ex_{k,\ell}|+|\ex_{k,\ell}'|$ for each $k,\ell$; so $Q^{j,\pm}_\ex$ are homogeneous of trace degree $|\ex|$.  The above argument shows that
\[ \sum_{\xi\in\beta_N^\pm} \del_{\xi_j}^2 [v_\ex]_N = [Q^{j,\pm}_\ex]_N, \]
and so setting $Q^j_\ex = rQ^{j,+}_\ex + sQ^{j,-}_\ex$ completes item (1) of the theorem.

For item (2), fix $\EX_J \ni \d = ((h_1,\d_1),\ldots,(h_{p},\d_{p}))$; then $[v_\d]_N(A_1,\ldots,A_n) = \tr(A^{\d_1}_{h_1}\cdots A^{\d_{p}}_{h_p})$.  Thus, for $\xi\in\beta_N$,
\begin{equation} \label{e.diff3} (\del_\xi [v_\ex]_N)(\del_\xi [v_\d]_N) = \sum_{k=1}^m\sum_{\ell=1}^p \delta_{j,j_k}\delta_{j,h_\ell}\tr(A_{j_1}^{\ex_1}\cdots (A_{j_k}\xi)^{\ex_k}\cdots A_{j_m}^{\ex_m})\tr(A_{h_1}^{\d_1}\cdots (A_{h_\ell}\xi)^{\d_\ell} \cdots A_{h_p}^{\d_p}). \end{equation}
(To be clear: the terms $\delta_{j,j_k}\delta_{j,h_\ell}$ are indicator functions, not related to the string $\delta\in\EX_J$.)  Taking $\del_{i\xi_j}$ instead yields the same formula, possibly with some minus signs inside the sum (depending on $\ex_k$ and $\d_\ell$).  We can write each term in (\ref{e.diff3}) in the form
\[ \pm \tr(\xi A^{\ex^{(k)}})\tr(\xi A^{\d^{(\ell)}}) \]
where $\ex^{(k)}$ and $\d^{(\ell)}$ are certain cyclic permutations of $\ex$ and $\d$.  Using (\ref{e.ma2}), summing over $\xi\in\beta_N^\pm$ then yields
\[ \frac{1}{N^2}\sum_{k=1}^m\sum_{\ell=1}^p \pm \delta_{j,j_k}\delta_{j,h_\ell} [v_{\ex^{(k)}\d^{(\ell)}}]_N, \]
where $\ex^{(k)}\d^{(\ell)}$ denotes the concatenation; in particular, $|\ex^{(k)}\d^{(\ell)}| = |\ex|+|\d|$.  Thus, setting
\[ R^{j,\pm}_{\ex,\d} = \sum_{k=1}^m\sum_{\ell=1}^p \pm \delta_{j,j_k}\delta_{j,h_\ell} v_{\ex^{(k)}\d^{(\ell)}} \]
(where the $\pm$ on the two sides do not necessarily match), we have shown that
\[ \sum_{\xi\in\beta_N^\pm} (\del_{\xi_j}[v_\ex]_N)(\del_{\xi_j}[v_\d]_N) = \frac{1}{N^2}[R^{j,\pm}_{\ex,\d}]_N. \]
Set $R^j_{\ex,\d} \equiv r R^{j,+}_{\ex,\d} + s R^{j,-}_{\ex,\d}$; then $R^j_{\ex,\d}$ has homogeneous trace degree $|\ex|+|\d|$, and so satisfies item (2), concluding the proof of the theorem.  \end{proof}

\begin{theorem}[Intertwining Formula] \label{t.intertwine1} For $j\in J$, let $\left\{Q^j_\ex\colon\ex\in\EX_J\right\}$ and $\left\{R^j_{\ex,\d}\colon \ex,\d\in\EX_J\right\}$ be the collections in $\PP(J)$ given in Theorem \ref{t.intertwine0}.  Define the following operators on $\PP(J)$:
\begin{equation} \D^j_{r,s} = \sum_{\ex\in\EX_J} Q^j_{\ex}\frac{\del}{\del v_\ex} \qquad \text{and} \qquad \L^j_{r,s} = \sum_{\ex,\d\in\EX_J} R^j_{\ex,\d}\frac{\del^2}{\del v_\ex \del v_\d}. \end{equation}
Then $\D^j_{r,s}$ and $\L^j_{r,s}$ preserve trace degree (when $(r,s)\ne(0,0)$), and, for all $P\in\PP(J)$,
\begin{equation} \label{e.intertwine1} \Delta^{j,N}_{r,s}([P]_N) = \left[\left(\D^j_{r,s} + \frac{1}{N^2}\L^j_{r,s}\right)P\right]_N \end{equation}
\end{theorem}

\begin{proof} The proof is almost identical to the proof of \cite[Theorem 3.26]{DHK2013}; we repeat it here.  Let $\mx{V}_N\colon\GL_N^n\to \M_N^{\EX_J}$ be the map
\[ \left(\mx{V}_N(A_1,\ldots,A_n)\right)((j_1,\ex_1),\ldots,(j_m,\ex_m)) = \tr(A_{j_1}^{\ex_1}\cdots A_{j_m}^{\ex_m}). \]
Then, by definition, $[P]_N = P\circ\mx{V}_N$.  By the chain rule, if $\xi \in \gl_N$ then
\begin{align*}
\del_{\xi_j}^2 P_N = \del_{\xi_j}^2 (P\circ\mx{V}_N) &  = \sum_{\ex\in\EX} \del_{\xi_j}\left[\left(\frac{\del P}{\partial v_{\ex}}\right)(\mx{V}_N)\cdot  \del_{\xi_j}[v_\ex]_N\right]  \\
&= \sum_{\ex\in\EX}\left(\frac{\del P}{\del v_\ex}\right)(\mx{V}_N)\cdot  \del_{\xi_j}^{2}([v_\ex]_N)+ \sum_{\ex,\d\in\EX} \left(\frac{\del^2 P}{\del v_\ex \del v_\d}\right)(\mx{V}_N)\cdot \left(\partial_{\xi_j}[v_\ex]_N\right)\left(\del_{\xi_j} [v_\d]_N\right)
\end{align*}
from which it follows that
\begin{align*}
\Delta^{j,N}_{r,s} P_{N} &  = \sum_{\ex\in\EX} \left(\frac{\del P}{\del v_{\ex}}\right)(\mx{V}_N)\cdot  \Delta^{j,N}_{r,s}([v_\ex])\\
&  +\sum_{\ex,\d \in\EX} \left(\frac{\del ^2 P}{\partial v_{\ex}\partial v_{\d}}\right)(\mx{V}_N)\cdot  \left[ r\sum_{\xi\in\beta_N}(\del_{\xi_j}[v_\ex])(\del_{\xi_j}[v_\d]) + s \sum_{\xi\in i\beta_N} (\del_{\xi_j}[v_\ex])(\del_{\xi_j}[v_\d]_N)  \right].
\end{align*}
Combining this equation with the results of Theorem \ref{t.intertwine0} completes
the proof.  \end{proof}

This prompts us to define the following operators.
\begin{definition} Let $\mx{t}= (t_1,\ldots,t_n)$ for some $t_1,\ldots,t_n\ge 0$.  Define
\[ \D^\mx{t}_{r,s} = \frac12\sum_{j=1}^n t_j\D^j_{r,s}, \qquad \L^\mx{t}_{r,s} = \frac12\sum_{j=1}^n t_j\L^j_{r,s}. \] \end{definition}

\begin{corollary} \label{c.intertwine} For any $\mx{t}=(t_1,\ldots,t_n)\in\R_+^n$, and $d\in\N$,
$\D^\mx{t}_{r,s}$ and $\L^\mx{t}_{r,s}$ preserve the finite dimensional space $\PP_d(J)$, and
\[ e^{\frac12(t_1\Delta^{1,N}_{r,s}+\cdots +t_n\Delta^{n,N}_{r,s})} P_N = [e^{\D^\mx{t}_{r,s} + \frac{1}{N^2}\L^\mx{t}_{r,s}} P]_N, \qquad P\in\PP_d(J). \]
In particular, $e^{\D^\mx{t}_{r,s}+\frac{1}{N^2}\L^\mx{t}_{r,s}}$ and $e^{\D^\mx{t}_{r,s}}$ are well-defined operators on the space $\PP(J)$.
\end{corollary}

\begin{proof} Since $\D^j_{r,s}$ and $\L^j_{r,s}$ preserve trace degree, the corollary follows by expanding the exponentials as power series of operators acting on the finite dimensional spaces $\PP_d(J)$ and $[\PP_d(J)]_N$. \end{proof}

\begin{remark} Since $\Delta^{j,N}_{r,s}$ commute for $1\le j\le n$, it is natural to expect the same holds for the intertwining operators $\D^j_{r,s}$ and $\L^j_{r,s}$.  This is true, and follows easily from examining the explicit form of the coefficients of these operators given in Theorem \ref{t.intertwine0}.  One must be careful about drawing such conclusions in general, however; the map $P\mapsto P_N$ is generally not one-to-one, due to the Cayley-Hamilton Theorem.  It is {\em asymptotically} one-to-one, in the sense that its restriction to $\PP_d(J)$ is one-to-one for all sufficiently large $N$ (depending on $d$), and this can be used to prove this commutation result.  Note, however, that $[\D^j_{r,s},\L^j_{r,s}]\ne 0$ in general.  \end{remark}

\subsection{Concentration of Measure}

We restate a general linear algebra result here, given as \cite[Lemma 4.1]{DHK2013}.

\begin{lemma} \label{l.findim} Let $V$ be a finite dimensional normed $\C$-space and supposed that $D$ and $L$ are two operators on $V$. Then there exists a constant $C=C(D,L,\|\cdot\|_V)<\infty$ such that
\begin{equation} \label{e.findim1}\left\Vert e^{D+\epsilon L}-e^{D}\right\Vert _{\mathrm{End}(V)} \leq C\left\vert \epsilon\right\vert \text{ for all }\left\vert \epsilon
\right\vert \leq1, \end{equation}
where $\|\cdot\|_{\mathrm{End}(V)}$ is the operator norm on $V$. It follows that, if $\psi\in V^\ast$ is a linear functional, then
\begin{equation} \label{e.findim2} \left|\psi(e^{D+\epsilon L}x)-\psi(e^Dx)\right| \le C\|\psi\|_{V^\ast} \|x\|_V|\epsilon|, \quad x\in V, \; |\epsilon|\le 1, \end{equation}
where $\|\cdot\|_{V^\ast}$ is the dual norm on $V^\ast$.  \end{lemma}

Coupled with Corollary \ref{c.intertwine}, this gives the following.

\begin{proposition} \label{p.conc1} Let $P\in\PP(J)$.  Let $\mx{t}=(t_1,\ldots,t_n)\in\R_+^n$.  Then there is a constant $C=C(r,s,\mx{t},P)$ so that, for all $N\in\N$,
\[ \left|\int_{\GL_N^n} P_N\,d\mu^{n,N}_{r,s;\mx{t}} - \left(e^{\D^\mx{t}_{r,s}}P\right)(\mx{1})\right| \le \frac{C}{N^2}, \]
where, for $Q\in\PP(J)$, $Q(\mx{1})$ is the complex number given by evaluating all variables of $Q$ at $1$.
\end{proposition}

\begin{proof} Let $d=\deg(P)$; then $P\in\PP_d(J)$.  By definition (\ref{e.hkn}),
\[ \int_{\GL_N^n} P_N\,d\mu^{n,N}_{r,s;\mx{t}} = \left(e^{\frac12(t_1\Delta^{1,N}_{r,s} + \cdots + t_n\Delta^{n,N}_{r,s})}P_N\right)(I_N^n). \]
(To be clear: the function $P_N$ is not compactly-supported, so this does not fall strictly into the purview of (\ref{e.hkn}); that the formula extends to such trace polynomials follows from Langland's Theorem; cf.\ \cite[Theorem 2.1 (p.\ 152)]{Robinson1991}.  See \cite[Appendix A]{DHK2013} for a concise sketch of the proof.)  From Corollary \ref{c.intertwine}, therefore
\[ \int_{\GL_N^n} P_N\,d\mu^{n,N}_{r,s;\mx{t}} = \left[e^{\D^\mx{t}_{r,s}+\frac{1}{N^2}\L^\mx{t}_{r,s}}P\right]_N(I_N^n) = \left(e^{\D^\mx{t}_{r,s}+\frac{1}{N^2}\L^\mx{t}_{r,s}}P\right)(\mx{1}). \]
Note that $\psi_\mx{1}(P) = P(\mx{1})$ is a linear functional on the finite dimensional space $\PP_d(J)$; thus the result follows from (\ref{e.findim2}) by choosing any norm $\|\cdot\|_{\PP_d(J)}$ on $V=\PP_d(J)$, and setting
\[ C(r,s,\mx{t},P) = C(\D^\mx{t}_{r,s},\L^\mx{t}_{r,s},\|\cdot\|_{\PP_d(J)})\|\psi_\mx{1}\|_{\PP_d(J)^\ast}\|P\|_{\PP_d(J)}, \]
thus concluding the proof.  \end{proof}

We now come to the main theorems of this section.

\begin{theorem} \label{t.limdist} Let $(B^{1,N}_{r,s}(t))_{t\ge 0},\ldots,(B^{n,N}_{r,s}(t))_{t\ge 0}$ be independent Brownian motions on $\GL_N$.  Then these matrix processes have a  joint limit distribution: for any $m\in\N$, $j_1,\ldots,j_m\in\{1,\ldots,n\}$, $t_1,\ldots,t_n\ge 0$ and $\ex_1,\ldots,\ex_m\in\{1,\ast\}$,
\[ \lim_{N\to\infty} \E\tr(B^{{j_1},N}_{r,s}(t_{j_1})^{\ex_1}\cdots B^{{j_m},N}_{r,s}(t_{j_m})^{\ex_m}) \qquad \text{exists.} \]
\end{theorem}

\begin{proof} Let $\mx{t}=(t_1,\ldots,t_n)$.  The given expected trace is computed in terms of the joint law $\mu^{n,N}_{r,s;\mx{t}}$ of the independent Brownian random matrices as
\[ \E\tr(B^{j_1,N}_{r,s}(t_{j_1})^{\ex_1}\cdots B^{j_m,N}_{r,s}(t_{j_m})^{\ex_m}) = \int_{\GL_N^n} \tr(A_{j_1}^{\ex_1}\cdots A_{j_m}^{\ex_m})\,\mu^{n,N}_{r,s;\mx{t}}(dA_1\cdots dA_n) = \int_{\GL_N^n} [v_{\ex}]_N\,d\mu^{n,N}_{r,s;\mx{t}} \]
where $\ex = \left((j_1,\ex_1),\ldots,(j_m,\ex_m)\right)$.  Proposition \ref{p.conc1} thus shows that the limit as $N\to\infty$ exists, and is equal to $\big(e^{\D^\mx{t}_{r,s}}v_\ex\big)(\mx{1})$.
\end{proof}

\begin{theorem} \label{t.cov0} Let $P,Q\in\PP(J)$, and let $\mx{t}=(t_1,\ldots,t_n)\in\R_+^n$.  There is a constant $C_2 = C_2(r,s,\mx{t},P,Q)$ such that
\begin{equation} \label{e.cov0} \left|\Cov_{\mu^{n,N}_{r,s;\mx{t}}}(P_N,Q_N)\right| \le \frac{C_2}{N^2}. \end{equation}
\end{theorem}
\noindent Theorem \ref{t.cov0} is a generalization of \cite[Proposition 4.13]{Kemp2013a}, and the proof is very similar.  First, we need a lemma on intertwining complex conjugation, which is elementary to prove and left to the reader; cf.\  \cite[Lemma 3.11]{Kemp2013a}.

\begin{lemma} \label{l.CC} Given $\ex\in\EX_J$, define $\ex^\ast\in\EX_J$ by $\left((j_1,\ex_1),\ldots,(j_n,\ex_n)\right)^\ast = \left((j_n,\ex_n^\ast),\ldots,(j_1,\ex_1^\ast)\right)$ where $1^\ast = \ast$ and $\ast^\ast = 1$.  Define $\CC\colon\PP(J)\to\PP(J)$ to be the conjugate linear homomorphism satisfying $\CC(v_\ex) = v_{\ex^\ast}$ for all $\ex\in\EX_J$.  Then for all $N\in\N$
\begin{equation} \label{e.CC} \overline{P_N} = [\CC(P)]_N, \qquad P\in\PP(J). \end{equation}
%and, moreover, $\CC$ commutes with $\D^j_{r,s}$ and $\L^j_{r,s}$ for $1\le j\le n$.
\end{lemma}

\begin{proof}[Proof of Theorem \ref{t.cov0}] The covariance of $\C$-valued random variables $F,G$ is $\Cov(F,G) = \E(F\overline{G})-\E(F)\E(\overline{G})$.  Define $\D^{\mx{t},N}_{r,s} = \D_{r,s}^\mx{t} + \frac{1}{N^2}\L_{r,s}^\mx{t}$.  From Lemma \ref{l.CC}, we may write $P_N\overline{Q_N} = [PQ^\ast]_N$, and so, from (\ref{e.hkn}) and Corollary \ref{c.intertwine}, we have
\begin{equation} \label{e.covterm1} \E_{\mu^{n,N}_{r,s;\mx{t}}}(P_N\overline{Q}_N) = \left(e^{\D_{r,s}^{\mx{t},N}}(PQ^\ast)\right)(\mx{1}). \end{equation}
Similarly,
\begin{equation} \label{e.covterm2} \E_{\mu^{n,N}_{r,s;\mx{t}}}(P_N)\cdot \E_{\mu^{n,N}_{r,s;\mx{t}}}(\overline{Q_N}) = \left(e^{\D_{r,s}^{\mx{t},N}}P\right)(\mx{1})\cdot \left(e^{\D_{r,s}^{\mx{t},N}}Q^\ast\right)(\mx{1}). \end{equation}
Now, set
\begin{align} \label{e.Psi1} \Psi^N_1 \equiv \big(e^{-\D^{\mx{t},N}_{r,s}}P\big)(\mx{1}), \quad \Psi^N_\ast \equiv \big(e^{-\D^{\mx{t},N}_{r,s}}Q^\ast\big)(\mx{1}), \quad \Psi^N_{1,\ast} \equiv \big(e^{-\D^{\mx{t},N}_{r,s}}(PQ^\ast)\big)(\mx{1}), \\ \label{e.Psi2}
\Psi_1 \equiv \big(e^{-\D^\mx{t}_{r,s}}P\big)(\mx{1}), \quad\;\; \Psi_\ast \equiv \big(e^{-\D^\mx{t}_{r,s}}Q^\ast\big)(\mx{1}), \quad\;\; \Psi_{1,\ast} \equiv \big(e^{-\D^\mx{t}_{r,s}}(PQ^\ast)\big)(\mx{1}). \end{align}
Thus, (\ref{e.covterm1}) and (\ref{e.covterm2}) show that
\begin{equation} \label{e.ncvar3} \Cov_{\mu^{n,N}_{r,s;\mx{t}}}(P_N,Q_N)= \Psi^N_{1,\ast}-\Psi^N_1\Psi^N_\ast. \end{equation}
We estimate this as follows.  First
\begin{equation} \label{e.ncvar4} |\Psi^N_{1,\ast}-\Psi^N_1\Psi^N_\ast| \le |\Psi^N_{1,\ast}-\Psi_{1,\ast}| + |\Psi_{1,\ast}-\Psi_1\Psi_\ast| + |\Psi_1\Psi_\ast - \Psi_1^N\Psi^N_\ast|. \end{equation}
Referring to (\ref{e.Psi2}), note that $\D^{\mx{t}}_{r,s}$ is a first-order differential operator; it follows that $e^{\D^\mx{t}_{r,s}}$ is an algebra homomorphism, and so the second term in (\ref{e.ncvar4}) is $0$.  The first term is bounded by $\frac{1}{N^2}\cdot C(r,s,\mx{t},PQ^\ast)$ by Proposition \ref{p.conc1}.  For the third term, we add and subtract $\Psi_1^N\Psi_\ast$ to make the additional estimate
\begin{align} \nonumber |\Psi_1\Psi_\ast - \Psi_1^N\Psi_\ast^N| &\le |\Psi_\ast||\Psi_1-\Psi_1^N| + |\Psi_1^N||\Psi_\ast-\Psi_\ast^N| \\ \nonumber
&\le |\Psi_\ast||\Psi_1-\Psi_1^N|  + \big(|\Psi_1| + |\Psi_1^N-\Psi_1|)|\Psi_\ast-\Psi^N_\ast| \\ \nonumber
&\le \frac{1}{N^2}\cdot |\Psi_\ast|C(r,s,\mx{t},P) + \left(|\Psi_1| + \frac{1}{N^2}\cdot C(r,s,\mx{t},P)\right)\cdot\frac{1}{N^2}\cdot C(r,s,\mx{t},Q^\ast) \\ \label{e.ncvar5}
&= \frac{1}{N^2}\cdot\left(|\Psi_\ast|C(r,s,\mx{t},P)+|\Psi_1|C(r,s,\mx{t},Q^\ast)\right) + \frac{1}{N^4}\cdot C(r,s,\mx{t},P)C(r,s,\mx{t},Q^\ast).
\end{align}
Combining (\ref{e.ncvar5}) with (\ref{e.ncvar3}) -- (\ref{e.ncvar4}) and the following discussion shows that the constant
\begin{equation} \label{e.C_2} C_2(r,s,\mx{t},P,Q) = C(r,s,\mx{t},PQ^\ast) + C(r,s,\mx{t},P)C(r,s,\mx{t},Q^\ast) + |\Psi_\ast|C(r,s,\mx{t},P)+|\Psi_1|C(r,s,\mx{t},Q^\ast) \end{equation}
verifies (\ref{e.cov0}), proving the proposition.  \end{proof}

This brings us to the proof of Theorem \ref{t.multi-cov}.  For convenience, we restate that the desired estimate is
\begin{equation} \label{e.multi-cov2} \Cov\!\left[\tr(f(B^{1,N}_{r,s}(t_1),\ldots,B^{n,N}_{r,s}(t_n)^\ast)),\tr(g(B^{1,N}_{r,s}(t_1),\ldots,B^{n,N}_{r,s}(t_n)^\ast))\right] \le \frac{C_2}{N^2}, \end{equation}
for any $f,g\in\C\langle X_1,\ldots,X_n,X_1^\ast,\ldots,X_n^\ast\rangle$, for some constant $C_2=C_2(r,s,\mx{t},f,g)$; here $B^{1,N}_{r,s}(\cdot),\ldots,B^{n,N}_{r,s}(\cdot)$ are independent $(r,s)$-Brownian motions on $\GL_N$.

\begin{proof}[Proof of Theorem \ref{t.multi-cov}] Setting $\mx{t}=(t_1,\ldots,t_n)$, the covariance in (\ref{e.multi-cov2}) is precisely
\[ \Cov_{\mu^{n,N}_{r,s;\mx{t}}}\left([\Upsilon(f)]_N,[\Upsilon(g)]_N\right) \]
and so the result follows immediately from Theorem \ref{t.cov0}.  \end{proof}

Theorem \ref{t.multi-cov}, in the special case $f=g$, implies that the convergence to the joint limit distribution in Theorem \ref{t.limdist} is, in fact, almost sure.

\begin{corollary} Let $(B^{1,N}_{r,s}(t))_{t\ge 0},\ldots,(B^{n,N}_{r,s}(t))_{t\ge 0}$ be independent Brownian motions on $\GL_N$.  Then for any $t_1,\ldots,t_n\ge 0$ and any $f\in\C\langle X_1,\ldots,X_n,X_1^\ast,\ldots,X_n^\ast\rangle$, the random variable $\tr(f(B^{1,N}_{r,s}(t_n),\ldots,B^{n,N}_{r,s}(t_n)^\ast)$ converges to its mean almost surely. \end{corollary}
\noindent This follows immediately from the $O(1/N^2)$ covariance estimate of Theorem \ref{t.multi-cov}, together with Chebyshev's inequality and the Borel-Cantelli lemma.

Finally, we note that we have proven asymptotic freeness of independent $(r,s)$-Brownian motions.

\begin{corollary} \label{c.asympfree2} Let $t_1,\ldots,t_n\ge 0$ and let $B^{1,N}_{r,s}(t_1),\ldots,B^{n,N}_{r,s}(t_n)$ be independent random matrices sampled from $(r,s)$-Brownian motion.  Then these random matrices are asymptotically free. \end{corollary}

\begin{proof} As pointed out in Remark \ref{r.UNinvB}, the distribution of each $B^{j,N}_{r,s}(t_j)$ is invariant under $\U_N$-conjugation.  Theorems \ref{t.multi-cov} and \ref{t.limdist} then confirm all of the conditions of Theorem \ref{t.asympfree}, which demonstrates the asymptotic freeness as claimed.  \end{proof}

\section{Moment Calculations \label{section Moment Calculations}}

This section is devoted to the proof of Theorem \ref{t.moments}.  We begin by reiterating the following differential characterization of the constants $\nu_n(t)$ from (\ref{e.nuNu0}).

\begin{lemma} \label{l.identnu} Let $\{\nu_n\colon n\ge 0\}$ be the functions in (\ref{e.nuNu0}), and let $\varrho_n(t) = e^{\frac{n}{2}t}\nu_n(t)$.  The functions $\varrho_n$ are uniquely determined by the initial conditions $\varrho_n(0)=\nu_n(0) = 1$ for all $n$, $\varrho_1(t)\equiv 1$, and the following system of coupled linear ODEs for $n\ge 2$:
\[ \varrho_n'(t) = -\sum_{k=1}^{n-1} k \varrho_k(t)\varrho_{n-k}(t). \]
\end{lemma}
\noindent Indeed, in \cite{Biane1997c}, this connection was the key step in identifying the distribution of a free unitary Brownian motion as the limit distribution (at each fixed time $t$) of a Brownian motion $U^N_t$ on $\U_N$.  It is also independently proved in \cite[Lemma 5.4, Eq.\ (5.23)]{DHK2013}.

\begin{lemma} \label{l.da} Let $b_{r,s}(t)$ be defined by (\ref{e.fSDE0}); for short, let $b=b_{r,s}(t)$.  Set $a= a_{r,s}(t) = e^{\frac12(r-s)t}b$.  Then
\begin{equation} \label{e.da} da = a\,dw,
\end{equation}
where $w=w_{r,s}(t)$ of (\ref{e.w}).
\end{lemma}

\begin{proof} Since $t\mapsto e^{\frac12(r-s)t}$ is a free It\^o process with $de^{\frac12(r-s)t} = \frac12(r-s)e^{\frac12(r-s)t}\,dt$, (\ref{e.fItoprod}) shows that
\[ da = de^{\frac12(r-s)t}\cdot b + e^{\frac12(r-s)t}\cdot db + de^{\frac12(r-s)t}\cdot db. \]
The last term is $0$, while the first two simplify to
\[ da = \frac12(r-s)e^{\frac12(r-s)t}b\,dt + e^{\frac12(r-s)t}(b\,dw - \frac12(r-s)b\,dt) = a\,dw, \]
by (\ref{e.fSDE0}).
\end{proof}

We also record the following It\^o formula for $dw_{r,s}(t)$ products.
\begin{lemma} \label{l.dwdw*} Let $t\ge 0$ and let $\ex,\ex'\in\{1,\ast\}$.  For any adapted process $\theta=\theta(t)$,
\begin{equation} \label{e.dw.dw} dw^{\ex}\,\theta\,dw^{\ex'} = (s\pm r)\t(\theta)\,dt, \end{equation}
where the sign is $-$ if $\ex=\ex'$ and $+$ if $\ex\ne\ex'$.
\end{lemma}
\noindent Lemma \ref{l.dwdw*} is an immediate computation from (\ref{e.fIto2}) -- (\ref{e.fIto4}).

\subsection{The Moments of $b_{r,s}(t)$}

We use (\ref{e.da}) to give a recursive formula for the powers of $a_{r,s}(t)$.

\begin{proposition} \label{p.a} For $n\in\N^\ast$,
\begin{equation} \label{e.a1} d(a^n) = \sum_{k=1}^n a^k\,dw\, a^{n-k} + (s-r)\1_{n\ge 2}\sum_{k=1}^{n-1} ka^k\t(a^{n-k})\,dt. \end{equation}
\end{proposition}

\begin{proof} When $n=1$, (\ref{e.a1}) reduces to (\ref{e.da}).  We proceed by induction, supposing that (\ref{e.a1}) has been verified up to level $n$.  Then, using the It\^o product rule (\ref{e.fItoprod}), together with (\ref{e.da}) and (\ref{e.a1}), gives
\begin{align*} d(a^{n+1}) = d(a\cdot a^n) &= da\cdot a^n + a\cdot d(a^n) + da\cdot d(a^n) \\
&= a\,dw\,a^n + \sum_{k=1}^n a^{k+1}\,dw\,a^{n-k} + (s-r)\sum_{k=1}^{n-1} ka^{k+1}\t(a^{n-k})\,dt + \sum_{k=1}^n a\,dw\,a^k\,dw\,a^{n-k}.
\end{align*}
The first two terms combine, reindexing $\ell=k+1$, to give $\sum_{\ell=1}^{n+1} a^\ell\,dw\,a^{n+1-\ell}$. From (\ref{e.dw.dw}), the last terms are
\[ (s-r)\sum_{k=1}^n \t(a^k)a^{n+1-k}\,dt \]
which, when combined with the penultimate terms, yields (\ref{e.a1}) at level $n+1$.  This concludes the inductive proof.
\end{proof}

\begin{corollary} \label{c.moments1} The moments of $a=a_{r,s}(t)$ are $\t(a^n) = \varrho_n((r-s)t)$; consequently, the moments of $b=b_{r,s}(t)$ are $\t(b^n) = \nu_n((r-s)t)$, verifying (\ref{e.m1}).
\end{corollary}

\begin{proof} Since $a(0)=b(0)=1$, $\t(a(0)^n) = 1 = \varrho_n(0)$.  Taking the trace of (\ref{e.a1}) and using (\ref{e.fIto1}), we have
\begin{equation} \label{e.dtau(a)} d\t(a^n) = (s-r)\1_{n\ge 2}\sum_{k=1}^{n-1} k\t(a^k)\t(a^{n-k})\,dt. \end{equation}
Thus $\frac{d}{dt}\t(a) = 0 = \varrho_1'((r-s)t)$.  If $s=r$, (\ref{e.dtau(a)}) asserts that $\tau(a^n)=\tau(a(0)^n) = 1 = \varrho_n(0\cdot t)$ for all $n$.  On the other hand, if $s\ne r$, let $\tilde\varrho_n(t) = \t(a_{r,s}(t/(r-s))^n)$; then the chain rule applied to (\ref{e.dtau(a)}) shows that
\[ \tilde\varrho_n'(t) = -\1_{n\ge 2}\sum_{k=1}^{n-1} k\tilde{\varrho}_k(t)\tilde{\varrho}_{n-k}(t). \]
By Lemma \ref{l.identnu}, it follows that $\tilde{\varrho}_n(t) = \varrho_n(t)$ for all $n,t\ge 0$.  Hence, $\t(a_{r,s}(t)^n) = \varrho_n((r-s)t) = e^{\frac{n}{2}(r-s)t}\nu_n((r-s)t)$, as claimed.  As defined in Lemma \ref{l.da}, we therefore have
\[ \t(b^n) = \t[(e^{-\frac12(r-s)t}a)^n] = e^{-\frac{n}{2}(r-s)t}\varrho_n((r-s)t) = \nu_n((r-s)t), \]
verifying (\ref{e.m1}), and concluding the proof.
\end{proof}

\subsection{The Moments of $b_{r,s}(t)b_{r,s}(t)^\ast$}

\begin{lemma} \label{l.c0} Let $c_{r,s}(t) = e^{-st}b_{r,s}(t)$; for short, let $c=c_{r,s}(t)$.  Then
\begin{equation} \label{e.dcc*} d(cc^\ast) = 2\sqrt{s}\,c\,dy\,c^\ast, \end{equation}
where $y=y(t)$.
\end{lemma}

\begin{proof} First note that $cc^\ast = e^{-2st}bb^\ast$. As in Lemma \ref{l.da}, we have
\begin{equation} \label{e.dcc*0} d(cc^\ast) = -2s\,cc^\ast\,dt + e^{-2st}d(bb^\ast). \end{equation}
By the It\^o product rule (\ref{e.fItoprod}) and (\ref{e.fSDE0}),
\begin{align*} d(bb^\ast) &= db\cdot b^\ast + b\cdot db^\ast + db\cdot db^\ast \\
&= b\,dw\,b^\ast - \frac12(r-s)bb^\ast\,dt + b\,dw^\ast\,b^\ast -\frac12(r-s)bb^\ast\,dt + b\,dw\,dw^\ast\,b^\ast \\
&= b(dw+dw^\ast)b^\ast - (r-s)bb^\ast + (r+s)bb^\ast
\end{align*}
where the last equality follows from Lemma \ref{l.dwdw*}.  Note that $dw+dw^\ast = 2\sqrt{s}\,dy$, and so this simplifies to $d(bb^\ast) = 2\sqrt{s}b\,dy\,b^\ast + 2s\,bb^\ast\,dt$.  Combining this with (\ref{e.dcc*0}) yields the result.
\end{proof}

\begin{proposition} \label{p.cc*n} For $n\in\N^\ast$,
\begin{equation} \label{e.cc*n} d[(cc^\ast)^n] =   2\sqrt{s}\sum_{k=1}^n (cc^\ast)^{k-1}c\,dy\,c^\ast(cc^\ast)^{n-k} + 4s\1_{n\ge 2}\sum_{k=1}^{n-1} k(cc^\ast)^k\t[(cc^\ast)^{n-k}]\,dt. \end{equation}
\end{proposition}

\begin{proof} When $n=1$, (\ref{e.cc*n}) reduces to (\ref{e.dcc*0}), so we proceed by induction: suppose that (\ref{e.cc*n}) has been verified up to level $n$.   Then we use the It\^o product formula (\ref{e.fItoprod}), together with (\ref{e.dcc*0}) and (\ref{e.cc*n}), to compute
\begin{align*} d[(cc^\ast)^{n+1}] &= d(cc^\ast)\cdot (cc^\ast)^n + cc^\ast\cdot d[(cc^\ast)^n] + d(cc^\ast)\cdot d[(cc^\ast)^n] \\
&= 2\sqrt{s}\, c\,dy\,c^\ast(cc^\ast)^n +2\sqrt{s}\sum_{k=1}^n (cc^\ast)^{k}c\,dy\,c^\ast(cc^\ast)^{n-k} +4s\sum_{k=1}^{n-1} k(cc^\ast)^{k+1}\t[(cc^\ast)^{n-k}]\,dt \\
&\hspace{3.25in} + 4s\sum_{k=1}^n c\,dy\,c^\ast(cc^\ast)^{k-1}c\,dy\,c^\ast(cc^\ast)^{n-k}.
\end{align*}
Reindexing $\ell=k+1$, the first two terms combine to give $2\sqrt{s}\sum_{\ell=1}^{n+1} (cc^\ast)^{\ell-1}c\,dy\,c^\ast(cc^\ast)^{n+1-\ell}$.  In the last term, we use (\ref{e.fIto2}) to yield
\[ dy\,c^\ast(cc^\ast)^{k-1}c\,dy = \t(c^\ast(cc^\ast)^{k-1}c)\,dt = \t[(cc^\ast)^k]\,dt. \]
Hence, reindexing $j=n+1-k$, the final sum is
\[ 4s\sum_{k=1}^n \t[(cc^\ast)^k](cc^\ast)^{n+1-k}\,dt = 4s\sum_{j=1}^n (cc^\ast)^j \t[(cc^\ast)^{n+1-j}]. \]
Also reindexing the penultimate sum with $\ell=k+1$, the last two sums combine to give
\[ 4s\sum_{\ell=2}^n (\ell-1)(cc^\ast)^\ell \t[(cc^\ast)^{n+1-\ell}]\,dt + 4s\sum_{j=1}^n (cc^\ast)^j \t[(cc^\ast)^{n+1-j}]. \]
Note that the first sum could just as well be started at $\ell=1$ (since that term is $0$), and these two combine to give the second term in (\ref{e.cc*n}), concluding the inductive proof. \end{proof}

\begin{corollary} \label{cor.bb*} The moments of $cc^\ast$ are $\t[(cc^\ast)^n] = \varrho_n(-4st)$; consequently, the moments of $bb^\ast$ are $\t[(bb^\ast)^n] = \nu_n(-4st)$, verifying (\ref{e.m3}). \end{corollary}

\begin{proof}  Since $b(0)=1$, $\t[(cc^\ast(0))^n]=1=\varrho_n(0)$ for all $n$.  Taking the trace of (\ref{e.cc*n}), we have
\begin{equation} \label{e.cor.bb*1} d\t[(cc^\ast)^n] = 4s\1_{n\ge 2}\sum_{k=1}^{n-1} k\t[(cc^\ast)^k]\t[(cc^\ast)^{n-k}]\,dt. \end{equation}
Thus $\frac{d}{dt}\t(cc^\ast) = 0 = \varrho'_1(-4st)$.  If $s=0$, (\ref{e.cor.bb*1}) asserts that $\t[(cc^\ast)^n]= \t[(cc^\ast(0))^n] = 1 = \varrho_n(0\cdot t)$ for all $n$.  If $s\ne 0$, let $\hat\varrho_n(t) = \t[((cc^\ast)(-t/4s))^n]$; then the chain rule applied to (\ref{e.cor.bb*1}) shows that
\[ \hat\varrho_n'(t) = -\1_{n\ge 2}\sum_{k=1}^{n-1} k \hat\varrho_k(t)\hat\varrho_{n-k}(t). \]
By Lemma \ref{l.identnu}, it follows that $\hat{\varrho}_n(t) = \varrho_n(t)$ for all $n,t\ge 0$.  Hence,
\[ \t[(cc^\ast)^n)] = \varrho_n(-4st) = e^{\frac{n}{2}(-4s)t}\nu_n(-4st), \]
as claimed.  As defined in Lemma \ref{l.c0}, we therefore have
\[ \t[(bb^\ast)^n] = \t[(e^{2st}cc^\ast)^n] = e^{-2nst}\varrho_n(-4st) = \nu_n(-4st), \]
verifying (\ref{e.m2}), and concluding the proof.  \end{proof}

\subsection{The Trace of $b_{r,s}(t)^2 b_{r,s}(t)^{\ast 2}$}

Finally, we calculate $\t(b^2b^{\ast 2})$.  To that end, we need the following cubic moment as part of the recursive computation.

\begin{lemma} \label{l.daaa*} Let $a= e^{\frac12(r-s)t}b$ as in Lemma \ref{l.da}.  Then
\begin{equation} \label{e.daaa*.new} \t(a^2a^\ast) = (1+2st)e^{(s+r)t}. \end{equation}
\end{lemma}

\begin{proof} From the It\^o product rule (\ref{e.fItoprod}), we have
\[ d(a^2a^\ast) = da\cdot aa^\ast + a\cdot da\cdot da^\ast + a^2 da^\ast + (da)^2\cdot a^\ast + da\cdot a\cdot da^\ast + a\cdot da\cdot da^\ast. \]
Lemma \ref{l.da} asserts that $da=a\,dw$.  To compute $d\t(a^2a^\ast)$, we can ignore the first three terms that have trace $0$ by (\ref{e.fIto1}); the last three terms become
\[ a\,dw\,a\,dw\,a^\ast + a\,dw\,a\,dw^\ast\,a^\ast + a^2\,dw\,dw^\ast\,a^\ast = (s-r)\t(a)aa^\ast\,dt + (s+r)\t(a)aa^\ast\,dt + (s+r)a^2a^\ast\,dt \]
by Lemma \ref{l.dwdw*}.  Taking traces, we therefore have
\begin{equation} \label{e.daaa*DE1} d\t(a^2a^\ast) = 2s\t(a)\t(aa^\ast)\,dt + (s+r)\t(a^2a^\ast)\,dt. \end{equation}
In Corollary \ref{c.moments1}, we computed that $\t(a) = \varrho_1((r-s)t) = e^{\frac12(r-s)t}\nu_1((r-s)t)$, which, referring to (\ref{e.nuNu0}), is equal to $1$.  Similarly, in Corollary \ref{cor.bb*}, we calculated that $\t(bb^\ast) = \nu_1(-4st) = e^{2st}$, and so $\t(aa^\ast) = e^{(r-s)t}\t(bb^\ast) = e^{(r+s)t}$.  Hence, (\ref{e.daaa*DE1}) reduces to the ODE
\[ \frac{d}{dt}\t(a^2a^\ast) = 2se^{(r+s)t} + (s+r)\t(a^2a^\ast), \qquad \t(a^2a^\ast(0)) = 1. \]
It is simple to verify that (\ref{e.daaa*.new}) is the unique solution of this ODE.  \end{proof}
\begin{remark} As a sanity check, note that in the case $(r,s)=(1,0)$ (\ref{e.daaa*.new}) shows that $\t(b^2b^\ast) = e^{-\frac{3}{2}t}\t(a^2a^\ast) = e^{-t/2}$.  As pointed out in (\ref{e.fubm}), $b_{1,0}(t) = u(t)$ is a free unitary Brownian motion, and so $\t(b^2b^\ast) = \t(b)$ in this case; thus, we have consistency with (\ref{e.nuNu0}).
\end{remark}

\begin{proposition} \label{p.bbb*b*}  Let $a= e^{\frac12(r-s)t}b$ as in Lemma \ref{l.da}.  Then
\begin{equation} \label{e.aaa*a*} \t(a^2a^{\ast 2}) = 4st(1+st)e^{(s+r)t} + e^{2(s+r)t} \end{equation}
and thus (\ref{e.m3}) holds true.
\end{proposition}

\begin{proof} Expanding, once again, using the It\^o product rule (\ref{e.fItoprod}), we have
\begin{align} \label{e.daaa*a*1} d(a^2a^{\ast 2}) &= da\cdot aa^{\ast 2} + a\cdot da\cdot a^{\ast 2} + a^2\cdot da^\ast\cdot a^\ast + a^2a^\ast\cdot da^\ast \\
\label{e.daaa*a*2} & \qquad + (da)^2\cdot a^{\ast 2} + da\cdot a\cdot da^\ast\cdot a^\ast + da\cdot aa^\ast\cdot da^\ast \\
\label{e.daaa*a*3} & \qquad + a\cdot da\cdot da^\ast\cdot a^\ast + a\cdot da\cdot a^\ast\cdot da^\ast + a^2\cdot (da^\ast)^2.
\end{align}
The terms in (\ref{e.daaa*a*1}) all have trace $0$.  We simplify the terms in (\ref{e.daaa*a*2}) and (\ref{e.daaa*a*3}) using $da = a\,dw$ and Lemma \ref{l.dwdw*} as follows:
\begin{align*} \text{(\ref{e.daaa*a*2})} &= a\,dw\,a\,dw\,a^{\ast 2} + a\,dw\,a\,dw^\ast\,a^{\ast 2} + a\,dw\,aa^\ast\,dw^\ast\,a^\ast \\
&= (s-r)\t(a)aa^{\ast 2}\,dt + (s+r)\t(a)aa^{\ast 2}\,dt + (s+r)\t(aa^\ast)aa^\ast\,dt,
\end{align*}
and
\begin{align*} \text{(\ref{e.daaa*a*3})} &= a^2\,dw\,dw^\ast\,a^{\ast 2} + a^2\,dw\,a^\ast\,dw^\ast\,a^\ast + a^2\,dw^\ast\,a^\ast\,dw^\ast\,a^\ast \\
&= (s+r)a^2a^{\ast 2}\,dt + (s+r)\t(a^\ast)a^2a^\ast\,dt + (s-r)\t(a^\ast)a^2a^\ast\,dt.
\end{align*}
Taking traces, and using the fact (from Lemma \ref{l.daaa*}) that $\t(a^\ast)\t(a^2a^\ast)$ is real, this yields
\[ d\t(a^2a^{\ast 2}) = 2s\t(a)\t(aa^{\ast 2})\,dt + (s+r)[\t(aa^\ast)]^2\,dt + (s+r)\t(a^2a^{\ast 2})\,dt + 2s\t(a^\ast)\t(a^2a^\ast)\,dt. \]
Using (\ref{e.daaa*.new}), together with (\ref{e.m2}) and the fact (pointed out in the proof of Lemma \ref{l.daaa*}) that $\t(a)= 1$, gives
\begin{equation} \label{e.finalODE} \frac{d}{dt}\t(a^2a^{\ast 2}) = 4s(1+2st)e^{(s+r)t} + (s+r)e^{2(s+r)t}+(s+r)\t(a^2a^{\ast 2}). \end{equation}
It is easy to verify that (\ref{e.aaa*a*}) is the unique solution to this ODE with initial condition $1$.  Substituting $b = e^{\frac12(s-r)t}a$ then yields (\ref{e.m3}).
\end{proof}

\begin{remark} Again, as a sanity check, (\ref{e.m3}) reduces to $\t(b^2b^{\ast 2}) = 1$ when $s=0$; this is consistent with the fact that $b$ is unitary in this case.  \end{remark}

\section{Properties of the Brownian Motions \label{section Properties of the Brownian Motions}}

Theorem \ref{t.brst} summarizes the main properties of both the matrix Brownian motions $B_{r,s}^N(t)$ on $\GL_N$ and its limit $(b_{r,s}(t))_{t\ge 0}$.  We will prove these properties separately for finite $N$ versus the limit, although in many cases the proofs are extremely similar.

\subsection{Properties of $(B^N_{r,s}(t))_{t\ge 0}$}

We begin by noting that the invertibility of $B^N_{r,s}(t)$ follows from the mSDE (\ref{e.mSDE-B}).

\begin{proposition} \label{p.B1} The diffusion $B^N_{r,s}(t)$ is invertible for all $t\ge 0$ (with probability $1$); the inverse $B^N_{r,s}(t)^{-1}$ is a {\em right}-invariant version of an $(r,s)$-Brownian motion.
\end{proposition}

\begin{proof} Fix a Brownian motion $W_{r,s}^N(t) = \sqrt{r}\,iX^N(t) + \sqrt{s}\,Y^N(t)$ on $\gl_N$, so that $B^N_{r,s}(t)$ is the solution of (\ref{e.mSDE-B}) with respect to $W_{r,s}^N(t)$.  Then define $A^N_{r,s}(t)$ to be the solution to
\begin{equation} \label{e.mSDE-Binv} dA^N_{r,s}(t) = -dW_{r,s}^N(t)\,A^N_{r,s}(t)-\frac12(r-s)A^N_{r,s}(t)\,dt. \end{equation}
Note that $-X^N(t)$ and $-Y^N(t)$ are also independent $\mathrm{GUE}_N$ Brownian motions, so $A^N_{r,s}(t)$ is a right-invariant version of $B^N_{s,t}(t)$.  (Indeed, the reader can readily check that, if $\del_\xi$ is replaced with the right-invariant derivative $\frac{d}{dt}f(\exp(-t\xi)g)$, thus defining a right-invariant Laplacian, the associated Brownian motion satisfies (\ref{e.mSDE-Binv}).)  To simplify notation, let $W=W^N_{r,s}(t)$, $B=B^N_{r,s}(t)$, and $A=A^N_{r,s}(t)$.  Using the It\^o product rule (\ref{e.mItoprod}), we have
\begin{align*} d(BA)&= dB\cdot A + B\cdot dA + dB\cdot dA \\
&= B\,dW\,A - \frac12(r-s)BA\,dt -B\, dW\,A - \frac12(r-s)BA\,dt - B\,(dW)^2\,A.
\end{align*}
From (\ref{e.mIto2}) -- (\ref{e.mIto4}), we compute exactly as in Lemma \ref{l.dwdw*} that $(dW)^2 = (s-r)I_N\,dt$.  This shows that $d(BA) = 0$.  Since $B^N_{r,s}(0)=A^N_{r,s}(0) = I_N$, it follows that $BA = I_N$, so $A^N_{r,s}(t) = B^N_{r,s}(t)^{-1}$, as claimed.
\end{proof}

\begin{proposition} \label{p.B2} The multiplicative increments of $(B^N_{r,s}(t))_{t\ge 0}$ are independent and stationary.
\end{proposition}

\begin{proof} Let $0\le t_1<t_2<\infty$, and let $\mathscr{F}_{t_1}$ denote the $\sigma$-field generated by $\{X^N(t),Y^N(t)\}_{0\le t\le t_1}$.  From the defining mSDE (\ref{e.mSDE-B}), we have
\begin{equation*} B^N_{r,s}(t_2)-B^N_{r,s}(t_1) = \int_{t_1}^{t_2} B^N_{r,s}(t)\,dW^N_{r,s}(t) - \frac12(r-s)\int_{t_1}^{t_2}\,B^N_{r,s}(t)\,dt, \end{equation*}
or, in other words,
\begin{equation} \label{e.deltaB} B^N_{r,s}(t_1)^{-1}B^N_{r,s}(t_2) = I_N+ \int_{t_1}^{t_2} B^N_{r,s}(t_1)^{-1}B^N_{r,s}(t)\,dW^N_{r,s}(t) - \frac12(r-s)\int_{t_1}^{t_2}\,B^N_{r,s}(t_1)^{-1}B^N_{r,s}(t)\,dt. \end{equation} 
This shows that the process $C^N(t) = B_{r,s}^N(t_1)^{-1}B^N_{r,s}(t)$ for $t\ge t_1$ satisfies the mSDE
\[ dC^N(t) = C^N(t)\,d(W^N_{r,s}(t)-W^N_{r,s}(t_1)) - \frac12(r-s)C^N(t)\,dt. \]
Note that $W^N_{r,s}(t)-W^N_{r,s}(t_1) = \sqrt{r}\,i(X^N(t)-X^N(t_1)) + \sqrt{s}\,(Y^N(t)-Y^N(t_1))$.  Since $(X^N(t)-X^N(t_1))_{t\ge t_1}$ and $(Y^N(t)-Y^N(t_1))_{t\ge t_1}$ are independent $\mathrm{GUE}_N$ Brownian motions, and since $C^N_{t_1} = I_N$, it follows that $(C^N(t))_{t\ge t_1}$ is a version of $(B^N_{r,s}(t))_{t\ge 0}$.  This shows, in particular, that the multiplicative increments are stationary.  Moreover, (\ref{e.deltaB}) shows that $B^N_{r,s}(t_1)^{-1}B^N_{r,s}(t_2)$ is measurable with respect to the $\sigma$-field generated by the increments $(W^N_{r,s}(t)-W^N_{r,s}(t_1))_{t_1\le t \le t_2}$, which is independent from $\mathscr{F}_{t_1}$ (since the additive increments of $X^N(t)$ and $Y^N(t)$ are independent).  Since all the random matrices $B^N_{r,s}(t')$ with $t'\le t_1$ are $\mathscr{F}_{t_1}$-measurable, it follows that $(B^N_{r,s}(t))_{t\ge 0}$ has independent multiplicative increments, as claimed.
\end{proof}

\break

\begin{proposition} \label{p.B3} For $r,s>0$ and $N\ge 2$, with probability $1$, $B^N_{r,s}(t)$ is non-normal for all $t>0$. \end{proposition}

\begin{proof} Let $\M_N^{\nor}$ denote the set of normal matrices.  Let $\mathbb{D}_N$ denote the $2N$ (real) dimensional space of diagonal matrices in $\M_N$, and $\mathbb{T}_N\subset\U_N$ the $N$ (real) dimensional maximal torus of diagonal unitary matrices.  The map $\Phi\colon \mathbb{D}_N\times \U_N\to \M_N^{\nor}$ given by $\Phi(D,U) = UDU^\ast$ is smooth, and (by the spectral theorem) surjective.  Since $\Phi(D,U) = \Phi(D,TU)$ for any $T\in\mathbb{T}_N$, the map descends to a smooth surjection $\widetilde{\Phi}\colon \mathbb{D}_N\times \U_N/\mathbb{T}_N\to \M_N^{\nor}$.  It follows that
\[ \mathrm{dim}_\R(\M_N^{\nor}) \le \mathrm{dim}_\R(\mathbb{D}_N) + \mathrm{dim}_\R(\U_N/\mathbb{T}_N) = 2N + N^2-N = N^2+N. \]
Thus, as a submanifold of $\M_N$ (which has real dimension $2N^2$), $\mathrm{codim}_\R(\M_N^{\nor}) \ge 2N^2 -(N^2+N) = N^2-N$.  This is $\ge 2$ for $N\ge 2$.

 The manifold $\GL_N$ is an open dense subset of $\M_N$, and the generator $\Delta_{r,s}^N$ is easily seen to be a non-degenerate elliptic operator on $C^\infty(\M_N)$.  Thus, by the main theorem of \cite{Rama1998}, $\M_N^{\nor}$ is a polar set for the diffusion generated by $\frac12\Delta^N_{r,s}$; i.e.\ the hitting time of $\M_N^{\nor}$ for $(B^N_{r,s}(t))_{t\ge 0}$ is $+\infty$ almost surely.  This concludes the proof.
\end{proof}

\begin{remark} If $D$ is in the open dense subset of $\mathbb{D}_N$ with all eigenvalues distinct, then the stabilizer of $D$ in $\U_N$ is exactly equal to $\mathbb{T}_N$; thus the map $\widetilde{\Phi}$ above is generically a local diffeomorphism.  It follows that $\mathrm{dim}_\R(\M_N^{\nor}) = N^2+N$. \end{remark}

Propositions \ref{p.B1} -- \ref{p.B3} address the first half of Theorem \ref{t.brst}.  Let us also address Remark \ref{r.B=U} here.

\begin{proposition} \label{p.B4} For $r> 0$, $V^N(t) \equiv B^N_{r,0}(t/r)$ is Brownian motion on $\U_N$ with respect to the metric induced by the inner product $\langle \xi,\eta\rangle = -N\Tr(\xi\eta)$ on $\u_N$.  \end{proposition}

\begin{proof} Let $\beta_N$ be the basis for $\u_N$ defined in (\ref{e.beta_N}); then $\beta_N$ is orthonormal for the stated inner product.  From (\ref{e.B(t)Ito}) and (\ref{e.magic}), we see that, with $W^N(t) = \sum_{\xi\in\beta_N} B_\xi(t)\,\xi$, the Brownian motion $U^N(t)$ on $\U_N$ satisfies the mSDE
\[ dU^N(t) = U^N(t)\,dW^N(t) -\frac12 U^N(t)\,dt, \qquad U^N(0) = I_N. \]
(Note: the proof that this process takes values in $\U_N$ for all $t\ge 0$ follows much the same way as the proof of Proposition \ref{p.B1}.)  Note, as above, that $W^N(t) = iX^N(t)$ where $X^N(t)$ is a $\mathrm{GUE}_N$ Brownian motion.  Now, from (\ref{e.mSDE-B}), we compute that, for $r>0$,
\begin{align*} dV^N(rt) = dB^N_{r,0}(t) = \sqrt{r}\,i B^N_{r,0}(t)\,dX^N(t) - \frac12 r B^N_{r,0}(t)\,dt &= iB^N_{r,0}(t) dX^N(rt)-\frac12B^N_{r,0}\,d(rt) \\
&= i V^N(rt)\,dX^N(rt)-\frac12V^N(rt)\,d(rt), \end{align*}
using the standard space-time scaling of the Brownian motion $X^N(t)$ and the chain rule.  Thus $V^N(t)$ satisfies the same mSDE, with the same initial condition, as $U^N(t)$; this proves the proposition.
\end{proof}

\subsection{Properties of $(b_{r,s}(t))_{t\ge 0}$}

\begin{proposition} \label{p.b1} For all $r,s,t\ge 0$, the free multiplicative $(r,s)$-Brownian motion $b_{r,s}(t)$ is invertible; the inverse $a_{r,s}(t) = b_{r,s}(t)^{-1}$ satisfies the fSDE
\begin{equation} \label{e.binvSDE} da_{r,s}(t) = -dw_{r,s}(t)\,a_{r,s}(t)-\frac12(r-s)\,a_{r,s}(t)\,dt. \end{equation}
\end{proposition}

\begin{proof} The proof proceeds very similarly to the proof of Proposition \ref{p.B1}: using (\ref{e.fIto2}) -- (\ref{e.fIto4}) instead of (\ref{e.mIto2}) -- (\ref{e.mIto4}), we compute that $d(b_{r,s}(t)a_{r,s}(t)) = 0$, which shows, since $b_{r,s}(0) = a_{r,s}(0) = 1$, that $b_{r,s}(t)a_{r,s}(t)=1$.  In this infinite-dimensional setting, we must also verify that $a_{r,s}(t)b_{r,s}(t)=1$.  To that end, to simplify notation, let $a_t=a_{r,s}(t)$, $b_t=b_{r,s}(t)$, and $w_t=w_{r,s}(t)$.  Then we have
\begin{align*} d(a_tb_t) &= da_t\cdot b_t + a_t\cdot db_t + da_t\cdot db_t \\
&= -dw_t\,a_tb_t - \frac12(r-s)a_tb_t\,dt + a_tb_t\,dw_t - \frac12(r-s)a_tb_t\,dt - dw_t\,a_tb_t\,dw_t \\
&= [a_tb_t,dw_t] - (r-s)a_tb_t\,dt - dw_t\,a_tb_t\,dw_t.
\end{align*}
From Lemma \ref{l.dwdw*},
\[ dw_t\,a_tb_t\,dw_t  = (s-r)\t(a_tb_t). \]
Thus, $a_tb_t$ satisfies the fSDE
\[ d(a_tb_t) = [a_tb_t,dw_t] +(r-s)[a_tb_t-\t(a_tb_t)], \]
with initial condition $a_0b_0=1$.  Notice that the fSDE
$d\theta_t = [\theta_t,dw_t] + (r-s)[\theta_t-\t(\theta_t)]$ holds true for any constant process $\theta_t$; thus, with initial condition $\theta_0=1$ uniquely determining the solution, we see that $a_tb_t=1$ as well.  \end{proof}

\begin{proposition} \label{p.b2} The multiplicative increments of $(b_{r,s}(t))_{t\ge 0}$ are freely independent and stationary. \end{proposition}

\noindent The proof of Proposition \ref{p.b2} is virtually identical to the proof of Proposition \ref{p.B2}; one need only replace the $\sigma$-fields $\mathscr{F}_t$ with the von Neumann algebras $\A_t = W^\ast\{x(t'),y(t')\colon 0\le t'\le t\}$.

\begin{proposition} \label{p.b3} For $r\ge 0$ and $s>0$, $b_{r,s}(t)$ is non-normal for all $t>0$.  \end{proposition}

\begin{proof} Let $b_t = b_{r,s}(t)$; we compute that
\[ [b_t,b_t^\ast]^2 = (b_tb_t^\ast)^2 - b_t(b_t^\ast)^2b_t-b_t^\ast b_t^2 b_t^\ast + (b_t^\ast b_t)^2, \]
and so
\[ \t\left([b_t,b_t^\ast]^2\right) = 2\t[(b_tb_t^\ast)^2]-2\t[b_t^2(b_t^\ast)^2]. \]
We now use (\ref{e.nuNu0}), (\ref{e.m2}), and (\ref{e.m3}) to expand this:
\begin{align*} \t[(b_tb_t^\ast)^2]-\t[b_t^2(b_t^\ast)^2] &= \nu_2(-4st) - (e^{4st}+4st(1+st)e^{(3s-r)t}) \\
&= e^{4st}(1+4st) - (e^{4st}+4st(1+st)e^{(3s-r)t}) \\
&= 4st e^{3st}[e^{st}-(1+st)e^{-rt}].
\end{align*}
Since $r\ge 0$, $e^{-rt}\le 1$, and since $s>0$, $e^{st}>1+st$.  It follows that $\t([b_t,b_t^\ast]^2)>0$ for $t>0$, proving that $b_t$ is not normal.
\end{proof}

\begin{proposition} \label{p.b4} For $r>0$, $u(t) = b_{r,0}(t/r)$ is a free unitary Brownian motion. \end{proposition}

\begin{proof} This follows immediately from Proposition \ref{p.B4}, together with \cite[Theorem 1]{Biane1997c}.  Alternatively, we can see directly that (\ref{e.fSDE0}) reduces to $du(t) = i\,dx(t)-\frac12u(t)\,dt$ for $u(t) = b_{1,0}(t)$, which is the defining SDE of a (left) free unitary Brownian motion, and then do a time change computation as in the proof of Proposition \ref{p.B4} for $b_{r,0}(t/r)$.  \end{proof}

\section{Convergence of the Brownian Motions \label{section Convergence of the Brownian Motions}}

This final section is devoted to the proof of Theorem \ref{t.main}: that the process $(B^N_{r,s}(t))_{t\ge 0}$ converges in noncommutative distribution to the process $(b_{r,s}(t))_{t\ge 0}$.  We first show the convergence of the random matrices $B^N_{r,s}(t)$ for each fixed $t\ge 0$; the multi-time statement then follows from asymptotic freeness considerations.

\subsection{Convergence for a Fixed $t$}

We begin by noting the single-$t$ version of Theorem \ref{t.multi-cov}, which was proved in \cite[Proposition 4.13]{Kemp2013a}.  For any $r,s>0$ and $t\ge 0$, and any noncommutative polynomials $f,g\in\C\langle X,X^\ast\rangle$, there is a constant $C_{r,s}(t,f,g)$ such that
\begin{equation} \label{e.var1} \Cov\left[\tr\!\left(f(B^N_{r,s}(t),B^N_{r,s}(t)^\ast)\right),\tr\!\left(g(B^N_{r,s}(t),B^N_{r,s}(t)^\ast)\right)\right] \le \frac{C_{r,s}(t,f,g)}{N^2},  \end{equation} 
where $C_{r,s}(t,f,g)$ depends continuously on $t$.

We now proceed to prove the fixed-$t$ case of Theorem \ref{t.main}.  The idea is to compare the mSDE for $B^N_{r,s}(t)$ to the fSDE for $b_{r,s}(t)$, and inductively show that traces of $\ast$-moments differ by $O(1/N^2$), using (\ref{e.var1}).

\begin{theorem} \label{t.main-fixed-t} Let $r,s,t\ge 0$.  Let $n\in\N$ and let $\ex = (\ex_1,\ldots,\ex_n)\in\{1,\ast\}^n$.  Then there is a constant $C'_{r,s}(t,\ex)$ that depends continuously on $r,s,t$ so that
\begin{equation} \label{e.conv-t} \left|\E\tr\!\left(B^N_{r,s}(t)^{\ex_1}\cdots B^N_{r,s}(t)^{\ex_n}\right)- \t(b_{r,s}(t)^{\ex_1}\cdots b_{r,s}(t)^{\ex_n})\right|\le \frac{C'_{r,s}(t,\ex)}{N^2}. \end{equation}
\end{theorem}

\begin{proof} In the case $n=0$, (\ref{e.conv-t}) holds true vacuously with $C'_{r,s}(t,\emptyset) = 0$.  When $n=1$, as computed in (\ref{e.m1}) we have $\t(b_{r,s}(t)^{\ex_1}) = \nu_1((r-s)t)$, and so (\ref{e.conv-t}) follows immediately from \cite[Theorem 1.3]{Kemp2013a}.  From here, we proceed by induction: assume that (\ref{e.conv-t}) has been verified up to, but not including, level $n$.

Fix $\ex=(\ex_1,\ldots,\ex_n)\in\{1,\ast\}^n$.  Let $A^N_{r,s}(t) = e^{\frac12(r-s)t}B^N_{r,s}(t)$, so that, following precisely the proof of Lemma \ref{l.da} but using (\ref{e.mItoprod}) instead of (\ref{e.fItoprod}), we have
\begin{equation} \label{e.dA0} dA^N_{r,s}(t) = A^N_{r,s}(t)\,dW^N_{r,s}(t). \end{equation}
For convenience, denote $A=A^N_{r,s}(t)$, and denote $A^\ex = A^{\ex_1}\cdots A^{\ex_n}$.  Then, using the It\^o product rule (\ref{e.mItoprod}), we have
\begin{align} \label{e.dA1} d(A^\ex) &= \sum_{j=1}^n A^{\ex_1}\cdots A^{\ex_{j-1}}\cdot dA^{\ex_j}\cdot A^{\ex_{j+1}}\cdots A^{\ex_n} \\ \label{e.dA2}
&\qquad + \sum_{1\le j<k\le n} A^{\ex_1}\cdots A^{\ex_{j-1}}\cdot dA^{\ex_j}\cdot A^{\ex_{j+1}}\cdots A^{\ex_{k-1}}\cdot dA^{\ex_k}\cdot A^{\ex_{k+1}}\cdots A^{\ex_n}. \end{align}
From (\ref{e.mIto3}) and (\ref{e.dA0}), the terms in (\ref{e.dA2}) become
\begin{align*} &A^{\ex_1}\cdots A^{\ex_{j-1}}\cdot dA^{\ex_j}\cdot A^{\ex_{j+1}}\cdots A^{\ex_{k-1}}\cdot dA^{\ex_k}\cdot A^{\ex_{k+1}}\cdots A^{\ex_n} \\
=\;\;& A^{\ex_1}\cdots A^{\ex_{j-1}}A^{\ex_j'}\,dW^{\ex_j}\,A^{\ex_j''}A^{\ex_{j+1}}\cdots A^{\ex_{k-1}} A^{\ex_k'}\,dW^{\ex_k}\,A^{\ex_k''}A^{\ex_{k+1}}\cdots A^{\ex_n}
\end{align*}
where $W = W_{r,s}^N(t)$, and $1'=1$, $1''=\ast'=0$, and $\ast''=\ast$.  As in Lemma \ref{l.dwdw*}, (\ref{e.mIto2}) -- (\ref{e.mIto4}) show that, for any adapted process $\Theta$,
\begin{equation} \label{e.dWdW*} dW^{\ex}\,\Theta\,dW^{\ex'} = (s\pm r)\,\tr(\Theta)\,dt\, \end{equation}
where the sign is $-$ if $\ex=\ex'$ and $+$ is $\ex\ne\ex'$.  Hence, the terms in (\ref{e.dA2}) become
\[ (s\pm r) \tr(A^{\ex_j''}A^{\ex_{j+1}}\cdots A^{\ex_{k-1}}A^{\ex_k'}) A^{\ex_1}\cdots A^{\ex_{j-1}}A^{\ex_j'}A^{\ex_k''}A^{\ex_{k+1}}\cdots A^{\ex_n}. \]
Now, note that the expected value of all the terms in (\ref{e.dA1}) is $0$ by (\ref{e.mIto1}) and (\ref{e.dA0}).  Therefore, taking $\E\tr$ in (\ref{e.dA1}) and (\ref{e.dA2}), we have
\[ \frac{d}{dt}\E\tr(A^\ex) = \sum_{1\le j<k\le n} (s\pm r)\E\left[\tr(A^{\ex_j''}A^{\ex_{j+1}}\cdots A^{\ex_{k-1}}A^{\ex_k'}) \tr(A^{\ex_1}\cdots A^{\ex_{j-1}}A^{\ex_j'}A^{\ex_k''}A^{\ex_{k+1}}\cdots A^{\ex_n})\right]. \]
It is possible for one of the two trace terms to be trivial, in two special cases.
\begin{itemize}
\item If $j=1$ and $k=n$, and if $\ex_1=\ast$ and $\ex_n=1$, then the first trace term is equal to $\tr(A^\ex)$, while the second one is just $\tr(I_N) = 1$.
\item For $1\le j<n$, if $k=j+1$, and $\ex_j=1$ while $\ex_k=\ast$, then the second trace term is equal to $\tr(A^\ex)$, while the first one is just $\tr(I_N)=1$.
\end{itemize}
In all other $(\ex,j,k)$ configurations, each trace term involves a non-trivial string of length $<n$.  Note that, in both these exceptional cases, the two exponents must be different, and so the factor in front is $s+r$.  We separate out these cases as follows:
\begin{align*} \frac{d}{dt}\E\tr(A^\ex) &= (s+r)\1_{(\ex_1,\ex_n)=(\ast,1)}\E\tr(A^\ex) + (s+r)\sum_{j=1}^{n-1} \1_{(\ex_j,\ex_{j+1})=(1,\ast)} \E\tr(A^\ex) \\
&+ \widetilde{\sum_{1\le j< k\le n}} (s\pm r)\E\left[\tr(A^{\ex_j''}A^{\ex_{j+1}}\cdots A^{\ex_{k-1}}A^{\ex_k'}) \tr(A^{\ex_1}\cdots A^{\ex_{j-1}}A^{\ex_j'}A^{\ex_k''}A^{\ex_{k+1}}\cdots A^{\ex_n})\right],
\end{align*}
where $\widetilde{\sum}$ indicates that the sum excludes the at-most-$n$ terms accounted for in the special cases.  Define
\[ \kappa(\ex) = \1_{(\ex_1,\ex_n)=(\ast,1)} + \sum_{j=1}^{n-1} \1_{(\ex_j,\ex_{j+1})=(1,\ast)},\]
and let 
\[ \ex_{j,k}^1 = (\ex_j'',\ldots,\ex_k'), \qquad \ex_{j,k}^2 = (\ex_1,\ldots,\ex_j',\ex_k'',\ldots,\ex_n). \]
Thus we have shown that $\E\tr(A^\ex)$ satisfies the ODE
\begin{equation} \label{e.ODEA} \frac{d}{dt}\E\tr(A^\ex) = \kappa(\ex)(s+r)\E\tr(A^\ex) + \widetilde{\sum_{1\le j<k\le n}} (s\pm r)\E\left[\tr(A^{\ex^1_{j,k}})\tr(A^{\ex^2_{j,k}})\right], \end{equation}
where all the terms in the sum are expectations of products of traces of words in $A$ and $A^\ast$ of length {\em strictly less} than $n$. Since $A(0)=I_N$, the unique solution of this ODE (in terms of these functions in the sum) is
\begin{equation} \label{e.ODEA0}
\E\tr(A_T^\ex) = e^{\kappa(\ex)(s+r)T} + \widetilde{\sum_{1\le j<k\le n}} (s\pm r)\int_0^T e^{\kappa(\ex)(s+r)(T-t)} \E\left[\tr(A_t^{\ex^1_{j,k}})\tr(A_t^{\ex^2_{j,k}})\right]\,dt
\end{equation}
where we have written $A_t = A^N_{r,s}(t)$ to emphasize the different times of evaluation.  Now returning to $B_t = B^N_{r,s}(t) = e^{-\frac12(r-s)t}A_t$, and noting that the total length of the two strings $\ex^1_{j,k}$ and $\ex^2_{j,k}$ is $n$, the same as the length of $\ex$, this gives
\begin{align} \nonumber 
\E\tr(B_T^\ex) &= e^{[\kappa(\ex)(s+r) -\frac{n}{2}(r-s)]T} \\ \label{e.ODEB0}
&+  \widetilde{\sum_{1\le j<k\le n}} (s\pm r)\int_0^T e^{[\kappa(\ex)(s+r)-\frac{n}{2}(r-s)](T-t)}e^{\frac{n}{2}(r-s)t} \E\left[\tr(B_t^{\ex^1_{j,k}})\tr(B_t^{\ex^2_{j,k}})\right]\,dt.
\end{align}

Now, repeating this deviation line-by-line, we find that, setting $b_t = b_{r,s}(t)$,
\begin{align} \nonumber 
\t(b_T^\ex) &= e^{[\kappa(\ex)(s+r) -\frac{n}{2}(r-s)]T} \\ \label{e.ODEb0}
&+  \widetilde{\sum_{1\le j<k\le n}} (s\pm r)\int_0^T e^{[\kappa(\ex)(s+r)-\frac{n}{2}(r-s)](T-t)}e^{\frac{n}{2}(r-s)t} \t(b_t^{\ex^1_{j,k}})\t(b_t^{\ex^2_{j,k}})\,dt.
\end{align}
The principal difference is that, when applying the free It\^o product rule (\ref{e.fItoprod}), the trace $\t$ factors through completely, while in the matrix It\^o product rule (\ref{e.mItoprod}), only the trace $\tr$ factors through, while the expectation $\E$ does not.  Thus, the desired quantity (on the left-hand-side of (\ref{e.conv-t})) at time $T$ is equal to
\begin{equation} \label{e.B-b} \widetilde{\sum_{1\le j<k\le n}} (s\pm r)\int_0^T e^{[\kappa(\ex)(s+r)-\frac{n}{2}(r-s)](T-t)}e^{\frac{n}{2}(r-s)t} \left( \E\left[\tr(B_t^{\ex^1_{j,k}})\tr(B_t^{\ex^2_{j,k}})\right]-  \t(b_t^{\ex^1_{j,k}})\t(b_t^{\ex^2_{j,k}})\right)\,dt. \end{equation}
Again to simplify notation, fix $j,k$ in the sum and let $B^\ell = B^{\ex^\ell_{j,k}}_t$ and $b^\ell = b^{\ex^\ell_{j,k}}_t$ for $\ell=1,2$.  Then we expand the difference as
\begin{equation} \label{e.cov2} \E[\tr(B^1)\tr(B^2)]-\t(b^1)\t(b^2) = \Cov[\tr(B^1),\tr(B^2)] + \E\tr(B^1)\E\tr(B^2)-\t(b^1)\t(b^2), \end{equation}
and the last two terms may be written (by adding and subtracting $\t(b^1)\E\tr(B^2)$) as
\begin{equation} \label{e.cocov} \E\tr(B^1)\E\tr(B^2)-\t(b^1)\t(b^2) = \E\tr(B^2)\cdot[\E\tr(B^1)-\t(b^1)] + \t(b^1)\cdot[\E\tr(B^2)-\t(b^2)]. \end{equation}

We now appeal to the inductive hypothesis.  By construction, all the terms in the sum $\widetilde{\sum}$ have both strings $\ex_{j,k}^1$ and $\ex_{j,k}^2$ of length {\em strictly} $<n$.  As such, the inductive hypothesis yields that $|\E\tr(B^\ell)-\t(b^\ell)| \le C^\ell(t)/N^2$ for constants $C^\ell(t)$ that depend continuously on $t$ (and all of the hidden parameters $r,s,\ex$).  It follows, in particular, that the constants $\E\tr(B^2)$ are uniformly bounded in $N$ and $t\in[0,T]$.  Thus, the terms in (\ref{e.cocov}) are bounded by $C(t)/N^2$ for some constant $C(t)$ that is uniformly bounded in $t\in[0,T]$.  By (\ref{e.var1}), the covariance term in (\ref{e.cov2}) is also bounded by $C'(t)/N^2$ for such a constant $C'(t)$. Integrating $C(t)+C'(t)$ times the relevant exponentials, summed over $j,k$, in (\ref{e.B-b}) now shows that the whole expression is $\le C''(T)/N^2$ for some constant $C''(T)$ that depends continuously on $T$.  This concludes the proof.
\end{proof}

\begin{remark} In \cite[Theorem 1.6]{Kemp2013a}, the author showed that there exists a linear functional $\p_{r,s}^t\colon \C\langle X,X^\ast\rangle\to\C$ so that (\ref{e.conv-t}) holds with $\p_{r,s}^t(X^{\ex_1}\cdots X^{\ex_n})$ in place of $\t(b_{r,s}(t)^{\ex_1}\cdots b_{r,s}(t)^{\ex_n})$; the upshot of the present theorem is to identify this linear functional as the noncommutative distribution of $b_{r,s}(t)$.  In particular, it lives in a faithful, normal, tracial $W^\ast$-probability space, which could not be easily proved using the techniques in \cite{Kemp2013a}.
\end{remark}

%\begin{remark} Note that, by Proposition \ref{p.cov} and the Borel-Cantelli lemma, the convergence in Theorem \ref{t.main-fixed-t} is actually almost sure: that is,
%\[ \lim_{N\to\infty} \tr\!\left(B^N_{r,s}(t)^{\ex_1}\cdots B^N_{r,s}(t)^{\ex_n}\right) = \t(b_{r,s}(t)^{\ex_1}\cdots b_{r,s}(t)^{\ex_n}) \;\; a.s. \]
%\end{remark}

\subsection{Asymptotic Freeness and Convergence of the Process \label{subsection Asymp Freeness}} 

In this final section, we use the freeness of the increments of $b_{r,s}(t)$ and the asymptotic freeness of the increments of $B^N_{r,s}(t)$, together with Theorem \ref{t.main-fixed-t}, to prove Theorem \ref{t.main}.  We begin with some preliminary lemmas.

\begin{lemma} \label{l.time1} Let $\ex_1,\ldots,\ex_n\in\{1,\ast\}$, and let $f\in\C\langle X_1,\ldots,X_n\rangle$ be a noncommutative polynomial.  Given any permutation $\sigma\in\Sigma_n$, there is a noncommutative polynomial $g\in\C\langle X_1,\ldots,X_n,X_1^\ast,\ldots,X_n^\ast\rangle$ with the following property.  If $b_1,\ldots,b_n$ are invertible random variables in a noncommutative probability space, and $a_1=b_1, a_2 = b_1^{-1}b_2, \ldots, a_n = b_{n-1}^{-1}b_n$ are the corresponding multiplicative increments, then
\[ f(b_{\sigma(1)}^{\ex_1},\ldots,b_{\sigma(n)}^{\ex_n}) = g(a_1,\ldots,a_n,a_1^\ast,\ldots,a_n^\ast). \]
\end{lemma}

\begin{proof} For $1\le j\le n$, write
\begin{equation} \label{e.incr1} b_j = b_1(b_1^{-1}b_2)\cdots(b_{j-1}^{-1}b_j) = a_1a_2\cdots a_j. \end{equation}
Let $f_\sigma(X_1,\ldots,X_n) = f(X_{\sigma(1)},\ldots,X_{\sigma(n)})$; then
\[ f(b_{\sigma(1)}^{\ex_1},\cdots,b_{\sigma(n)}^{\ex_n}) = f_\sigma(b_1^{\ex_{\sigma^{-1}(1)}},\ldots,b_n^{\ex_{\sigma^{-1}(n)}}). \]
In each variable, expand the term $b_j^{\ex_{\sigma^{-1}(j)}}$ using (\ref{e.incr1}) (to the $\ex_{\sigma^{-1}(j)}$ power); this yields the polynomial $g$.
\end{proof}

The next lemma uses the language of Section \ref{section Trace Polynomials} to give a more precise formulation of how free independence reduces the calculation of joint moments to separate moments.

\begin{lemma} \label{l.freeness} Given any $n\in\N$ and any noncommutative polynomial $g\in\C\langle X_1,\ldots,X_n,X_1^\ast,\ldots,X_n^\ast\rangle$, there is an $m\in\N$ and a collection $\{P^{j,k}\colon 1\le j\le n, 1\le k\le m\}$ of elements of $\PP$ with the property that, if $(\A,\t)$ is a noncommutative probability space, and $a_1,\ldots,a_n\in\A$ are freely independent, then
\begin{equation} \label{e.freeness2} \t(g(a_1,\ldots,a_n,a_1^\ast,\ldots,a_n^\ast)) =  \sum_{k=1}^m P^{1,k}_\t(a_1)\cdots P^{n,k}_\t(a_n). \end{equation}
\end{lemma}
\noindent Here $\PP$ denotes the polynomial space $\PP(J)$ with the index set $J$ a singleton.  The proof of Lemma \ref{l.freeness} is contained in the proof of \cite[Lemma 5.13]{NicaSpeicherBook}.  The idea is to center the variables and proceed inductively.  The exact machinery of how $P^{j,k}$ are computed from $g$ is the business of the rich theory of free cumulants, which is the primary topic of the monograph \cite{NicaSpeicherBook}.

Now, suppose $A_1^N,\ldots,A_n^N$ are $N\times N$ random matrices that are asymptotically free; cf.\ Definition \ref{d.asympfree}.  This means precisely that $(A_1^N,\ldots,A_n^N)\to (a_1,\ldots,a_n)$ in noncommutative distribution, for some freely independent collection $a_1,\ldots,a_n$ in a noncommutative probability space $(\A,\t)$.  In other words, for any noncommutative polynomial $g\in\C\langle X_1,\ldots,X_n,X_1^\ast,\ldots,X_n^\ast\rangle$,
\begin{align*} \lim_{N\to\infty} \E\tr\!\left(g(A_1^N,\ldots,A_n^N,(A_1^N)^\ast,\ldots,(A_n^N)^\ast)\right) &= \t(g(a_1,\ldots,a_n,a_1^\ast,\ldots,a_n^\ast)) \\
&= \sum_{k=1}^m P^{1,k}_\t(a_1)\cdots P^{n,k}_\t(a_n)
\end{align*}
where the second equality is Lemma \ref{l.freeness}.  Note that $P^{j,k}_\t(a)$ is a polynomial in the trace moments of $a,a^\ast$, and by assumption of convergence of the joint distribution, we also therefore have $(P^{j,k}_{\E\tr}(A^N_j))\to P^{j,k}_\t(a_j)$ as $N\to\infty$.  Hence, we can alternatively state asymptotic freeness as
\begin{equation}\label{e.asympfree3} \lim_{N\to\infty} \E\tr\!\left(g(A_1^N,\ldots,A_n^N,(A_1^N)^\ast,\ldots,(A_n^N)^\ast)\right) = \lim_{N\to\infty}\sum_{k=1}^m P^{1,k}_{\E\tr}(A_1^N)\cdots P^{n,k}_{\E\tr}(A^N_n). \end{equation}

We now stand ready to prove Theorem \ref{t.main}.

\begin{proof}[Proof of \ref{t.main}] For convenience, denote $B^N_{r,s}(t) = B_t$ and $b_{r,s}(t) = b_t$.  Fix $t_1,\ldots,t_n\ge 0$ and $\ex_1,\ldots,\ex_n\in\{1,\ast\}$.  Fix a permutation $\sigma\in\Sigma_n$ such that $t_{\sigma(1)}\le \cdots \le t_{\sigma(n)}$ and let   $t_j' = t_{\sigma(j)}$. Let
\[ A_1 = B_{t_1'}, \; A_2 = B_{t_1'}^{-1}B_{t_2},\; \ldots,\; A_n = B_{t_{n-1}'}^{-1}B_{t_n}\]
be the increments for the partition $t_1'\le\cdots\le t_n'$.  Using Lemma \ref{l.time1}, we can write
\begin{equation} \label{e.final1} \E\tr\!\left(B_{t_1}^{\ex_1} \cdots B_{t_n}^{\ex_n}\right) = \E\tr(g(A_1,\ldots,A_n,A_1^\ast,\ldots,A_n^\ast)) \end{equation}
where $g\in\C\langle X_1,\ldots,X_n,X_1^\ast,\ldots,X_n^\ast\rangle$ is determined by $\sigma$ and $\ex_1,\ldots,\ex_n$.

By Proposition \ref{p.B2}, the increments $A_j$ are independent; moreover, their stationarity means that $A_j$ has the same distribution as $B_{\Delta t_j'}$ where $\Delta t_1' = t_1'$ and $\Delta t_j' = t_j'-t_{j-1}'$ for $1<j\le n$.  Thus, by Corollary \ref{c.asympfree2}, $A_1,\ldots,A_n$ are asymptotically free.  In addition, the equality of distributions means that all $\ast$-moments of $A_j$ are equal to the same $\ast$-moments of $B_{\Delta t_j'}$.  Thus, combining (\ref{e.asympfree3}) and (\ref{e.final1}), we have
\[ \lim_{N\to\infty}\E\tr(B_{t_1}^{\ex_1}\cdots B_{t_n}^{\ex_n}) = \lim_{N\to\infty}\sum_{k=1}^m P^{1,k}_{\E\tr}(B_{\Delta t_1'})\cdots P^{n,k}_{\E\tr}(B_{\Delta t_n'}). \]
From Theorem \ref{t.main-fixed-t}, we therefore have
\[ \lim_{N\to\infty}\E\tr(B_{t_1}^{\ex_1}\cdots B_{t_n}^{\ex_n}) = \sum_{k=1}^m P^{1,k}_\t(b_{\Delta t_1'})\cdots P^{n,k}_\t(b_{\Delta t_n'}). \]
Now, by Proposition \ref{p.b2}, the increments $b_{\Delta t_j'}$ are freely independent and stationary; so letting
\[ a_1 = b_{t'_1}, \; a_2 = b_{t'_1}^{-1}b_{t'_2}, \; \ldots, \; a_n = b_{t'_{n-1}}^{-1}b_{t'_n} \]
we see that $\{b_{\Delta t_1'},\ldots,b_{\Delta t_n'}\}$ have the same joint distribution as $\{a_1,\ldots,a_n\}$.  Thus
\[ \lim_{N\to\infty}\E\tr(B_{t_1}^{\ex_1}\cdots B_{t_n}^{\ex_n}) = \sum_{k=1}^m P^{1,k}_\t(b_{\Delta t_1'})\cdots P^{n,k}_\t(b_{\Delta t_n'}) = \sum_{k=1}^m P^{1,k}_\t(a_1)\cdots P^{n,k}_\t(a_n), \]
and by the definition (\ref{e.freeness2}) of $P^{j,k}$, this yields
\[ \lim_{N\to\infty}\E\tr(B_{t_1}^{\ex_1}\cdots B_{t_n}^{\ex_n}) = \t(g(a_1,\ldots,a_n,a_1^\ast,\ldots,a_n^\ast)). \]
Finally, by the definition (\ref{e.final1}) of $g$, we conclude that
\[ \lim_{N\to\infty}\E\tr(B_{t_1}^{\ex_1}\cdots B_{t_n}^{\ex_n}) = \t(b_{t_1}^{\ex_1}\cdots b_{t_n}^{\ex_n}), \]
concluding the proof.  \end{proof}

\subsection*{Acknowledgments} The author wishes to thank Philippe Biane, Bruce Driver, Pat Fitzsimmons, and Jamie Mingo for useful conversations.
\bibliographystyle{acm}
\bibliography{IBM}

\end{document}